\documentclass[a4paper,11pt]{amsart}

\usepackage{comment}

\usepackage{etoolbox}
\ifdef{\crop}{%
\usepackage[includeheadfoot,twoside=False,paperwidth=448pt,paperheight=587pt,rmargin=15pt,lmargin=15pt,tmargin=15pt,bmargin=15pt]{geometry}%
}{%
\setlength{\topmargin}{27mm}
\addtolength{\topmargin}{-1in}
\setlength{\oddsidemargin}{35mm}
\addtolength{\oddsidemargin}{-1in}
\setlength{\evensidemargin}{35mm}
\addtolength{\evensidemargin}{-1in}
\setlength{\textwidth}{140mm}
\setlength{\textheight}{220mm}
}%

\usepackage[active]{srcltx}
\usepackage{amsmath,amsthm,amsxtra}
\usepackage{amssymb}
\usepackage{mathabx}

\usepackage[mathscr]{eucal}

\usepackage{bm}
\usepackage[all]{xy}

\usepackage{aliascnt}

\theoremstyle{plain}

\newtheorem{thm}{Theorem}[section]
\newtheorem*{thm*}{Theorem}

\newaliascnt{prop}{thm}
\newaliascnt{cor}{thm}
\newaliascnt{lem}{thm}
\newaliascnt{claim}{thm}
\newaliascnt{defn}{thm}
\newaliascnt{ques}{thm}
\newaliascnt{conj}{thm}
\newaliascnt{fact}{thm}
\newaliascnt{rem}{thm}
\newaliascnt{ex}{thm}
\newtheorem{prop}[prop]{Proposition}
\newtheorem{cor}[cor]{Corollary}
\newtheorem{lem}[lem]{Lemma}
\newtheorem{claim}[claim]{Claim}
\newtheorem*{prop*}{Proposition}
\newtheorem*{cor*}{Corollary}
\newtheorem*{lem*}{Lemma}
\newtheorem*{claim*}{Claim}
\theoremstyle{definition}
\newtheorem{defn}[defn]{Definition}

\newtheorem*{defn*}{Definition}
\newtheorem*{ques*}{Question}
\newtheorem*{conj*}{Conjecture}

\newtheorem*{prob*}{Problem}

\newtheorem{rem}[rem]{Remark}
\newtheorem{ex}[ex]{Example}
\newtheorem*{fact*}{Fact}
\newtheorem*{rem*}{Remark}
\newtheorem*{ex*}{Example}
\aliascntresetthe{prop}
\aliascntresetthe{cor}
\aliascntresetthe{lem}
\aliascntresetthe{claim}
\aliascntresetthe{defn}
\aliascntresetthe{ques}
\aliascntresetthe{conj}
\aliascntresetthe{fact}
\aliascntresetthe{rem}
\aliascntresetthe{ex}

\usepackage[pointedenum]{paralist}
\usepackage{varioref}
\labelformat{equation}{\textnormal{(#1)}}
\labelformat{enumi}{\textnormal{(#1)}}

\usepackage[a4paper]{hyperref}

\def\textsectionN~{\textsection{}}

\renewcommand\phi{\varphi}
\renewcommand\epsilon{\varepsilon}
\renewcommand\leq{\leqslant}
\renewcommand\geq{\geqslant}

\makeatletter
\newcommand{\set}{  \@ifstar{\@setstar}{\@set}}\newcommand{\@setstar}[2]{\{\, #1 \mid #2 \,\}}
\newcommand{\@set}[1]{\{ #1 \}}
\makeatother
\newcommand{\lin}[1]{\langle #1 \rangle}
\newcommand{\trans}[1][1]{\raisebox{#1ex}{\scriptsize\kern0.1em$t$\kern-0.1em}}

\newcommand{\PP}{\mathbb{P}}
\newcommand{\cP}{{\PP_{\! *}}}
\newcommand{\TT}{\mathbb{T}}
\newcommand{\PN}{\PP^N}
\newcommand{\Pv}[1][N]{(\PP^{#1})\spcheck}

\newcommand{\ZZ}{\mathbb{Z}}

\newcommand{\CC}{\mathbb{C}}

\DeclareMathOperator{\Proj}{Proj}

\DeclareMathOperator{\Supp}{Supp}

\DeclareMathOperator{\Conv}{Conv}
\DeclareMathOperator{\rank}{rk}

\DeclareMathOperator{\id}{id}
\DeclareMathOperator{\Cone}{Cone}

\DeclareMathOperator{\Sym}{Sym}
\DeclareMathOperator{\GL}{GL}

\DeclareMathOperator{\pr}{pr}

\DeclareMathOperator{\Aff}{Aff}
\DeclareMathOperator{\Zero}{Zero}

\def\N{\mathbb{N}}
\def\Z{\mathbb{Z}}
\def\Q{\mathbb{Q}}
\def\R{\mathbb{R}}
\def\C{\mathbb{C}}

\def\r+{\mathbb{R}_{\geq 0}}

\def\r+{{\R}_{\geq 0}}
\def\q+{{\Q}_{\geq 0}}
\def\P{\mathbb{P}}
\def\arw{\rightarrow}
\def\*c{\C^{\times}}

\def\C{\mathbb {C}}

\def\N{\mathbb {N}}

\def\Q{\mathbb {Q}}
\def\R{\mathbb {R}}

\def\Z{\mathbb {Z}}

\newcommand{\cale}{\mathcal {E}}

\newcommand{\calo}{\mathcal {O}}

\newcommand{\Cc}{(\*c)^c}

\makeatletter
  
  \@addtoreset{equation}{section}
\makeatother

\title[On dual defects of toric varieties]{A combinatorial description of dual defects of toric varieties}

\author[K.~Furukawa]{Katsuhisa~FURUKAWA}
\address{Graduate School of Mathematical Sciences, 
the University of Tokyo, 
Tokyo, 
Japan}
\email{katu@ms.u-tokyo.ac.jp}

\author[A.~Ito]{Atsushi~Ito}
\address{Graduate School of Mathematics, Nagoya University,
Nagoya, Japan}
\email{atsushi.ito@math.nagoya-u.ac.jp}


\begin{document}

\maketitle

\begin{abstract}
From a finite set in a lattice,
we can define a toric variety embedded in a projective space.
In this paper,
we give a combinatorial description of the dual defect of the toric variety
using the structure of the finite set as a Cayley sum with suitable conditions.
We also interpret the description geometrically.
\end{abstract}

\section{Introduction}

Let $X \subset \PN$ be an $n$-dimensional projective variety over the field of complex numbers $\C$.
The \emph{dual variety} $X^* \subset  (\P^N)^{\vee }$ of $X$ is the closure of all points $H \in  (\P^N)^{\vee }$ 
such that as a hyperplane $H$ is tangent to $X$ at some smooth point $x \in X$,
where $(\P^N)^{\vee }$ is the dual projective space.
The \emph{dual defect} $\delta_X $ of $X$ is the natural number $N-1 - \dim X^*$.
In other words,
$\delta_X$ is the dimension of
the \emph{contact locus} on $X$ of general $H \in X^*$, which is defined to be
the closure of
\[
\{ x \in X_{sm} \, | \, H \text{ is tangent to } X \text{ at } x \},
\]
where $X_{sm} $ is the smooth locus of $X$.
A variety $X$ in $\P^N$ is said to be \emph{dual defective}
if the dual defect $\delta_X$ is positive.
Otherwise,
$X$ is called \emph{non-defective}.

In this paper, we consider the dual defect of a (not necessarily normal) toric variety defined from a finite subset of $\Z^n$ as follows (see \cite{GKZ}, \cite{CLS}). 

For a finite subset $A= \{ u_0,\ldots, u_N \} \subset \Z^n$,
we define the toric variety $X_A \subset \P^N$ to be the closure of the image of the morphism
\[
\phi_A : (\*c)^{ n} \arw \P^N \quad  : \quad  t  \mapsto [ t^{u_0} : \cdots : t^{u_N}],
\]
where $t^{u_i} := t_1^{u_{i}^1} \cdots t_n^{u_{i}^n}$ for
$t = (t_1, \dots, t_n)$ and $u_i = (u_{i}^1, \dots, u_{i}^n) \in \ZZ^n$.
We set $  \langle  A - A  \rangle \subset \Z^n $ 
to be the subgroup of $\Z^n$ 
generated by $A -A : = \{ u-u' \, | \, u,u' \in A \} $.
If $\langle  A - A  \rangle = \Z^n  $,
$(\*c)^{ n}$ is embedded into $X_A $ as an open dense subset by $\phi_A$.

Dual varieties of toric varieties have been studied by many authors from both the viewpoints of combinatorics and algebraic geometry,
starting with the work of Gelfand, Kapranov, and Zelevinsky \cite{GKZ}.
For example,
there are formulas to compute $\delta_{X_A} $ and the degree of the dual variety $X^*_{A}$ by \cite{GKZ}, \cite{DFS}, \cite{MT}, \cite{HS}.

The explicit structure of $A$ and $X_A$ for dual defective $X_A$ is also studied.
We recall the notion of Cayley sums.

\begin{defn}\label{def_cayley}
  Let $r \leq n$ be non-negative integers. Let $e_1,\ldots,e_r$ be the standard basis of $ \Z^r$.
  For finite sets $A_0,\ldots,A_r \subset \Z^{n-r}$,
  the Cayley sum $A_0 * \cdots * A_r  $ of $A_0,\ldots,A_r$ is defined to be
  \[
  A_0 * \cdots * A_r  := (A_0 \times \{0\}) \cup (A_1 \times \{e_1 \}) \cup \cdots \cup (A_r \times \{e_r\})  \subset \Z^{n-r} \times \Z^r.
  \]
  Throughout this paper, we always assume that $A_i \neq \emptyset$ for any $i$ when we consider the Cayley sum.
\end{defn}

Let $A \subset \Z^n$ be a finite subset with $\langle A-A \rangle =\Z^n$.
If $X_A $ is dual defective,
$X_A$ and $A$ have some special structure as follows.

If $X_A $ is smooth, Di Rocco \cite{DR} showed that
$X_A \subset \P^N$ has positive dual defect $d$ if and only if
there exist finite subsets $A_0, \dots, A_{r} \subset \Z^{n-r}$ with $r = (n+d) /2 $  
such that
\begin{itemize}
\item  
$A$ is $\Z$-affinely equivalent to the Cayley sum $A_0* \cdots *A_{r}$,
i.e.,
there exists a $\Z$-affine isomorphism $\Z^n \arw \Z^{n-r} \times \Z^r$ which maps $A$ onto $A_0* \cdots *A_{r}$,
\item all $\Conv(A_i)$'s have the same normal fan.
\end{itemize}
Geometrically,
this is equivalent to say that there exist a smooth toric variety $X'$ and torus equivariant ample line bundles $L_0,\ldots,L_r$ on $X'$ such that
\begin{align}\label{eq_proj_bundle}
(X_{A} ,\calo_{X_A}(1)) \simeq (\P_{X'}(L_0\oplus \cdots \oplus L_r) , \calo_{\P}(1)),
\end{align}
where $\calo_{\P}(1)$ is the tautological line bundle.
This result is generalized to the case when $X_A $ is normal and $\Q$-factorial by Casagrande and  Di Rocco \cite{CD}.
See also \cite{DN} for dual defects of smooth $X_A$.

On the other hand,
without any assumption on the singularity,
it is known that
$A$ is $\Z$-affinely equivalent to a Cayley sum $A_0 *A_{1}  $ 
if $X_A$ is dual defective by \cite{CC}, \cite{Es}.
More precisely,
$A$ is $\Z$-affinely equivalent to $A_0 * \cdots *A_{\delta} \subset  \Z^{n-\delta} \times  \Z^{\delta}$
for some $A_0,\ldots,A_{\delta} \subset  \Z^{n-\delta}$ for $\delta=\delta_{X_A}$ by {\cite[Corollary 4.2]{It}}.
However,
the converse statement 
does not hold in general (e.g., $X_A = \P^1 \times \P^1 \subset \P^3$ is non-defective for $A= \{0,1\} * \{0,1\} \subset \Z^1 \times \Z^1$).
To modify this result,
we consider the following condition on Cayley sums.

\begin{defn}\label{def_join_type}
  For finite sets $A_0,\ldots,A_r \subset \Z^{n-r}$,
 we say that the Cayley sum $A= A_0 * \cdots * A_r $ is \emph{of join type}
  if $\langle A_0 - A_0 \rangle  + \cdots + \langle A_r - A_r \rangle  \subset \Z^{n-r}$ is the inner direct sum of $\langle A_0 - A_0 \rangle  , \ldots, \langle A_r - A_r \rangle $,
  i.e.,
 \begin{align}\label{eq_of_join}
  \langle A_0 - A_0 \rangle \oplus  \cdots \oplus  \langle A_r - A_r \rangle \arw \Z^{n-r} \ : \  (m_0,\ldots,m_r) \mapsto m_0 + \cdots + m_r
\end{align}
  is injective.
  We note that \ref{eq_of_join}
  is surjective if and only if $\langle  A - A  \rangle =\Z^{n-r} \times \Z^r$ holds
  since $\langle  A - A  \rangle \cap (\Z^{n-r} \times \{0\} )$ coincides with the image of \ref{eq_of_join}.
  Hence under the assumption $ \langle A-A \rangle = \Z^{n-r} \times \Z^r$,
  $A_0 * \cdots * A_r $ is of join type if and only if \ref{eq_of_join} is an isomorphism.
  
      If $A= A_0 * \cdots * A_r $ is of join type,
    $X_A$ coincides with
    \[
    J(X_{A_0}, \dots, X_{A_r}) : = \bigcup_{x_0 \in X_0, \dots, x_r \in X_r} \overline{x_0 \dots x_r},
    \]
    the join of
    $X_{A_0}, \dots, X_{A_r}$ with natural embeddings
    $X_{A_i}  \hookrightarrow X_A \subset \PN$,
    where $\overline{x_0 \dots x_r}$ is
    the linear subvariety
    spanned by $x_0, \dots, x_r$.
See \autoref{subsec_join} for details.  
\end{defn}

Let $A=A_0 * \cdots *A_r \subset \Z^{n-r} \times \Z^r$ be a Cayley sum with $\langle  A - A  \rangle  =\Z^{n-r} \times \Z^r$
and let $ p : \Z^{n-r} \arw \Z^{n-r-c}$ be a surjective group homomorphism for some $ 0 \leq c \leq n-r$.
 The projection $\pi : \Z^{n-r} \times \Z^r \arw \Z^r$ to the second factor is decomposed as
\[
\Z^{n-r} \times \Z^r \stackrel{p \times \id_{\Z^r}}{\longrightarrow} \Z^{n-r-c} \times \Z^r \arw \Z^r.
\]
Since $p \times \id_{\ZZ_r} (A_0 * \cdots *A_r) = p(A_0) * \cdots * p(A_r) $ and $\pi(A_0 * \cdots *A_r) =\{0,e_1,\ldots,e_r\}$,
we have subvarieties of $X_{A_0 * \cdots *A_r} $ as
\[
X_{\{0,e_1,\ldots,e_r\}} \subset X_{  p(A_0) * \cdots * p(A_r) } \subset X_{A_0 * \cdots *A_r} .
\]
We note that $X_{\{0,e_1,\ldots,e_r\}}=\P^r$
and $X_{  p(A_0) * \cdots * p(A_r) } $ is of codimension $c$ in $X_{A_0 * \cdots *A_r}$.

\vspace{2mm}
The following is the main result of this paper.

\newcommand{\ThmStrStatement}{
  Let $A \subset \Z^n$ be a finite subset with $\langle  A - A  \rangle = \Z^n  $. 
  Then there exist finite subsets $A_0,\ldots,A_r \subset \Z^{n-r}$ and a surjective group homomorphism $ p : \Z^{n-r} \arw \Z^{n-r-c}$ 
  for some non-negative integers $r,c$
  such that
  \begin{itemize}
  \item[(a)] $A \subset \Z^n$ is $\Z$-affinely equivalent to the Cayley sum $A_0*\cdots* A_r \subset \Z^{n-r} \times \Z^{r}$,
  \item[(b)] $p(A_0) * \cdots * p(A_r) \subset \Z^{n-r-c} \times \Z^{r}$ is of join type,
\item[(c)] $\delta_{X_A}=r-c$,
\item[(d)] If $A'_0, \ldots, A'_{r'} \subset \Z^{n-r'}$ and $p' : \Z^{n-r'} \arw \Z^{n-r-c'}$ also satisfy (a), (b), (c) for some $r',c'$,
it holds that 
\[
X_{\{0,e_1,\ldots,e_r\}} \subset X_{\{0,e_1,\ldots,e_{r'}\}} \subset X_{  p'(A'_0) * \cdots * p'(A'_{r'}) } \subset  X_{  p(A_0) * \cdots * p(A_r) } 
\]
as subvarieties of $X_A$,
where we identify $X_A $ with $X_{A_0* \cdots * A_r}$ or $X_{A'_0* \cdots * A'_{r'}}$ by (a).
In particular,
$r \leq r'$ and $c \leq c'$ hold.
\end{itemize}
}

\begin{thm}
  \label{thm_structure}
  \ThmStrStatement{}
\end{thm}

The meaning of \autoref{thm_structure} is as follows:
If $A_0,\ldots,A_r \subset \Z^{n-r}$ and $ p : \Z^{n-r} \arw \Z^{n-r-c}$ satisfy (a), (b) in \autoref{thm_structure},
it is not so difficult to show $\delta_{X_A} \geq r-c$ (see \autoref{lem_tangent_dim_alpha} (iv)).
\autoref{thm_structure} states that the equality (c) holds for some $A_i$ and $p$.
As in the following example,
$A_i$ and $p$ which satisfy (a), (b), (c) are not unique in general.
Roughly speaking,
the property (d)  
guarantees the uniqueness of $A_i$ and $p$ in \autoref{thm_structure}
up to suitable $\Z$-affine isomorphisms. 
See also \autoref{rem_uniqueness}.


\begin{ex}\label{ex_not_unique}
Let $A=\{(0,0,0), (1,0,0), (0,1,0), (1,1,0), (0,0,1)\} \subset \Z^3$.
Then $A$ is $\Z$-affinely equivalent to the Cayley sum $A_0* A_1 \subset \Z^{2} \times \Z$
for 
\[
A_0=\{(0,0), (1,0), (0,1), (1,1)\} , \quad A_1=\{(0,0)\} \subset \Z^2.
\]
Hence $X_A \subset \P^4$ is the cone of $X_{A_0} =\P^1 \times \P^1 \subset \P^3$. 
In particular,
$\delta_{X_A}= 1$.
Since $A_0* A_1 $ is of join type,
$A_0,A_1$ and $p:=\id_{\Z^2} $ with $r=1,c=0$ satisfy (a), (b), (c) in \autoref{thm_structure}.
We can check (d) directly by considering all Cayley sums which are $\Z$-affinely equivalent to $A$.

We note that
$A$ is $\Z$-affinely equivalent to the Cayley sum $\tilde{A}_0* \tilde{A}_1* \tilde{A}_2 \subset \Z \times \Z^2$ for 
\[
\tilde{A}_0=\tilde{A}_1=\{0,1\}, \quad \tilde{A}_2=\{0\} \subset \Z.
\]
Then $\tilde{A}_0,\tilde{A}_1, \tilde{A}_2$ and the zero map $\tilde{p} :\Z \arw \Z^0=\{0\}$
with $\tilde{r}=2,\tilde{c}=1$ satisfy (a), (b), (c),
but do not satisfy (d).
\end{ex}

As a corollary of \autoref{thm_structure},
we obtain a combinatorial description of the dual defects of toric varieties.

\begin{cor}\label{thm_structure_max}
  Let $A \subset \Z^n$ be a finite subset with $\langle  A - A  \rangle = \Z^n  $. 
  Then the dual defect $\delta_{X_A} $ is the maximum of the values $ r-c $ for non-negative integers $r,c$ such that
  there exist finite subsets $A_0,\ldots,A_r \subset \Z^{n-r}$ and a surjective group homomorphism $ p : \Z^{n-r} \arw \Z^{n-r-c}$ 
which satisfy (a), (b) in \autoref{thm_structure}.
\end{cor}

\vspace{2mm}
Let
$A_0,\ldots,A_r \subset \Z^{n-r}$ and  $ p : \Z^{n-r} \arw \Z^{n-r-c}$ be as in \autoref{thm_structure}.
Since $p \times \id_{\ZZ_r} (A_0 * \cdots *A_r) = p(A_0) * \cdots * p(A_r) $ is of join type,
we have $X_{p \times 1_{\ZZ_r} (A_0 * \cdots *A_r)} = J(X_{p(A_0)}, \dots, X_{p(A_r)})$.
Furthermore,
$X_{p \times \id_{\ZZ_r} (A_0 * \cdots *A_r)} \subset X_A$ is nothing but the closure of a fiber of the torus equivariant rational map
$X_A \dashrightarrow (\*c)^c$ induced by  $\Z^c \simeq \ker (p \times \id_{\Z^r})  \subset \Z^{n-r} \times \Z^r$.
Hence we have
the following geometric structure of a toric variety
with respect to the dual defect.

\begin{thm}
  \label{thm_structure_geometry}
Let $A \subset \Z^n$ be a finite subset with $\langle  A - A  \rangle = \Z^n  $.
There exists a torus equivariant dominant rational map $\psi : X_A \dashrightarrow (\*c)^c$ for some $c \geq 0$ such that
the closure of each fiber is projectively equivalent
  to the join of $r+1$ non-defective toric varieties
  and $\delta_{X_A} = r-c$.
\end{thm}

In the general theory of projective geometry,
we have the following two lower bounds
about dual defects (see \autoref{sec_geometry}).

\begin{itemize}
\item Let $X \subset \PN$ be a projective variety.
If there exists a covering family $\{Y_{s }\}_{s}$ (i.e., $\overline{\bigcup Y_s} =X$) of subvarieties of $X$ of codimension $c$ and $\delta_{Y_{s}} = \delta$ for general $s$,
it holds that $\delta_X \geq \delta -c$.
\item If $Y$ is the join of $r+1$ varieties,
then $\delta_Y \geq r$ holds.
\end{itemize}
\autoref{thm_structure_geometry} means that $X_A$ has a covering family $\{ Y_{s } \}_{s \in (\*c)^r}$ for $Y_{s } = \overline{\psi^{-1}(s)} $
such that each $\overline{\psi^{-1}(s)} $ is the join of $r+1$ toric varieties
and the equalities hold in the above two inequalities.

\begin{rem}\label{rem_CD}
If $X_A$ is smooth,
we can take as $\psi$ in \autoref{thm_structure_geometry} the \emph{morphism} $\P_{X'}(L_0 \oplus \cdots \oplus L_r) \arw X'$ under the identification \ref{eq_proj_bundle}.
More generally,
if $X_A$ is normal and $\Q$-factorial,
we can take as $\psi$ an elementary extremal contraction of fiber type 
$\psi : X_A \arw X'$ to a normal $\Q$-factorial toric variety $X'$
by \cite[Corollary 5.5, Remark 5.6]{CD}.
However,
we cannot take a morphism as $\psi$ in general,
even if $X_A$ is normal. See \autoref{ex2}.
\end{rem}

In the end of the introduction,
we explain the idea of the proof of \autoref{thm_structure}.
Let $A \subset \Z^n$ be a finite set with $ \langle A-A\rangle=\Z^n$.
Let $H \subset \P^{N}$ be a general hyperplane which is tangent to $X_A$ at the unit element $1_n  \in (\*c)^{n}  \subset X_A$.
By definition, $ \delta_{X_A}$ is the dimension of the open subset of the contact locus,
\[
Z^{\circ}_H:= \{ x \in (\*c)^n \, | \, H \text{ is tangent to } X_A \text{ at } x \}  \subset X_A \subset \P^N.
\]

\vspace{1mm}
If $A=A_0*\cdots* A_r $ for some $A_0,\ldots,A_r \subset \Z^{n-r}$,
we have a subtorus 
\[
\{1_{n-r}\} \times (\*c)^r  \subset   (\*c)^{n-r}  \times (\*c)^{r}  \subset X_{A_0*\cdots* A_r} .
\]
By a direct calculation,
we have a linear algebraic description of $Z^{\circ}_H \cap ( \{1_{n-r}\} \times (\*c)^r )$ by using $\langle A_i -A_i \rangle $'s in $\Z^{n-r}$
in \autoref{sec_contact_locus}. 

\vspace{2mm}

To reduce \autoref{thm_structure} to this linear algebraic description,
we use the well-known fact that the closure $Z_H:=\overline{Z^{\circ}_H} \subset X_A \subset \P^N$ is a linear subvariety.
In \autoref{sec_refinement},
we show that the linear subvariety $Z_H \subset X_A$ induces an identification of $A$
with a Cayley sum $A_0*\cdots* A_r$ for some $r$
such that
$Z_H$ is \emph{contained} in a (possibly larger) linear subvariety
\[
\P^r = \overline{ \{1_{n-r}\} \times (\*c)^r } \subset X_{A_0*\cdots* A_r} =X_A.
\]
Then $Z^{\circ}_H$ is contained in $\{1_{n-r}\} \times (\*c)^r  $.
Hence we have a linear algebraic description of $Z^{\circ}_H =Z^{\circ}_H \cap ( \{1_{n-r}\} \times (\*c)^r )$,
from which we can construct $p : \Z^{n-r} \arw \Z^{n-r-c}$.

In the proofs of \cite{DR}, \cite{CD},
they construct $\psi$ in \autoref{thm_structure_geometry}  as an extremal contraction induced by lines in $Z_H$.
However,
we cannot take a morphism as $\psi$ in general.
This is the reason why we do not use extremal contractions in our proof.

\vspace{3mm}
This paper is organized as follows.
In \autoref{sec_notation},
we give some preliminaries.
In \autoref{sec_refinement},
we prove that a linear subvariety of $X_A$ induces an identification of $A$ with a suitable Cayley sum (\autoref{thm_plane_embedding}).
\autoref{thm_plane_embedding} has been proven concurrently by Ilten and Zotine in independent work \cite{IZ}.
In \autoref{sec_contact_locus},
we give a linear algebraic description of $Z^{\circ}_H \cap ( \{1_{n-r}\} \times (\*c)^r )$, as stated in the idea of the proof.
In \autoref{sec_proof},
we prove \autoref{thm_structure} and \autoref{thm_structure_max}. 
In \autoref{sec_geometry},
we interpret \autoref{thm_structure} geometrically and show \autoref{thm_structure_geometry}.
Throughout this paper,
we work over the field of complex numbers $\C$.

\subsection*{Acknowledgments}
The authors would like to thank Professor Nathan Ilten and Mr. Alexandre Zotine
for sending us their draft.
The second author was supported by Grant-in-Aid for JSPS fellows No.\ 14J01881.

\section{Preliminaries}\label{sec_notation}

\subsection{Notations and conventions}
We denote by $\N,\Z,\Q,\R$, and $\C$ the set of all
natural numbers, integers, rational numbers, real numbers, and complex numbers respectively.
In this paper, $\N$ contains $0$.
We denote $\C \setminus 0$ by $\*c$.
Let us denote by $1_n$ the unit element of the algebraic torus $(\*c)^n$.

Let  $M,M'$ be free abelian groups.
We say that a subset $S \subset M$ is $\Z$-\emph{affinely equivalent to} a subset $S' \subset M' $
if there exists a $\Z$-affine isomorphism $M \arw M'$ which maps $S$ onto $S'$.

For a free abelian group $M$ and for a field $k$,
we denote $M \otimes _{\Z} k$ by $M_{k}$.

Throughout this paper, we use the word ``general'' with respect to the Zariski topology,
that is,
for a variety $X$, we say that a property holds at a \emph{general} point of $X$
if it holds for all points in the complement of a proper algebraic subset.

\subsection{Preliminary on toric varieties and Cayley sums}\label{subsec_preliminary}

  For a finite subset $A \subset \Z^n$,
  let $\theta : \Z^m \arw  \Aff (A) $ be a $\Z$-affine isomorphism,
  where $\Aff (A) \subset \Z^n$ is the affine lattice spanned by $A$
  and $m= \rank \Aff (A) = \rank \langle A -A \rangle $.
  Then $ \theta^{-1}(A)$ spans $\Z^m$ as an affine lattice and
  $X_{\theta^{-1}(A)}$ is naturally identified with $X_A$
  by \cite[Chapter 5, Proposition 1.2]{GKZ}.
  In particular, we can identify $X_A $ and $X_{A'}$
  if $A \subset \Z^n$ and $A' \subset \Z^n$ are $\Z$-affinely equivalent.

\vspace{2mm}
In the rest of this subsection,
$A \subset \Z^n$ is a finite subset with $\langle A-A\rangle=\Z^n$.
The following is a well-known fact. See \cite[Lemma 2.1]{FI} for example for the proof.

\begin{lem}\label{lem_projection}
Let $\pi : \Z^n \arw \Z^{n'}$ be a surjective group homomorphism.
Then there exists a natural torus equivariant embedding $i : X_{\pi(A)} \hookrightarrow X_A \subset  \P^{\#A-1}$
so that $X_{\pi(A)} \subset \P^{\# \pi(A) -1}$ and $i (X_{\pi(A)}) \subset  \P^{\#A-1}$ are projectively equivalent.
\end{lem}


By this lemma,
we can regard $X_{\pi(A)}$ as a closed subvariety of $X_A$.
If $A=A_0 * \cdots * A_r$ is a Cayley sum and $\pi : \Z^{n-r} \times \Z^r \arw \Z^r$ is the projection to the second factor,
the subvariety $\P^r =X_{\{ 0,e_1,\ldots,e_r\}} \subset X_{A_0 * \cdots * A_r}$ coincides with 
the closure of the subtorus $ \{1_{n-r}\} \times (\*c)^r \subset (\*c)^{n-r} \times (\*c)^r \subset X_{A_0 * \cdots * A_r}$.

\vspace{2mm}

\begin{defn}\label{def_cayley_structure}
A surjective group homomorphism $\pi : \Z^n \arw \Z^r$ is called a \emph{Cayley structure} of $A$
if $\pi(A)$ is $\Z$-affinely equivalent to $\{ 0,e_1,\ldots,e_r\}$.

We say that two Cayley structures $\pi, \pi' :\Z^n \arw \Z^r$ of $A$ are \emph{equivalent}
if there exists a group isomorphism $\iota : \Z^r \arw \Z^r$ such that $\pi = \iota \circ \pi'$.
In other words,
$\pi$ and $\pi'$ are equivalent if and only if $X_{\pi(A)} $ coincides with $X_{\pi'(A)}$ as subvarieties of $X_A$.
\end{defn}

For example,
the projection $\Z^{n-r} \times \Z^r \arw \Z^r$ to the second factor is a Cayley structure of the Cayley sum $A_0 * \cdots * A_r \subset \Z^{n-r} \times \Z^r$.
The following lemma, which is also well-known,
states that
we can construct a Cayley sum from a Cayley structure conversely.
We give a proof since we will use the argument in \autoref{sec_proof}.

\begin{lem}\label{lem_projection_cayley}
Let $\pi : \Z^n \arw \Z^{r}$ be a Cayley structure of $A$.
Then there exist finite subsets $A_0,\ldots,A_r \subset \Z^{n-r}$ and $\Z$-affine isomorphisms $f : \Z^n \arw \Z^{n-r} \times \Z^r$ and $g: \Z^r \arw \Z^r$
such that $f(A)= A_0 * \cdots *A_r$ and
  \[
  \xymatrix{
    \Z^n \ar[d]_{f} \ar[r]^{\pi} \ar@{}[dr] &  \Z^r \ar[d]_{g} \\
     \Z^{n-r} \times \Z^r  \ar[r]^(.6){\pr_2} & \Z^r \\
  }
  \]
is commutative,
where $\pr_2$ is the projection to the second factor.
In particular,
\begin{itemize}
\item[(i)] $A$ is $\Z$-affinely equivalent to the Cayley sum $A_0 * \cdots * A_r$,
\item[(ii)] under the identification of $X_A $ and $X_{A_0 * \cdots * A_r} $ by $f$, the subvariety $X_{\pi(A)} \subset X_A$
coincides with $X_{\{ 0,e_1,\ldots,e_r\}} \subset X_{A_0 * \cdots * A_r}$.
\end{itemize}
\end{lem}

\begin{proof}
Let $s : \Z^r \arw \Z^n$ be a section of $0 \arw \ker \pi \arw \Z^n \stackrel{\pi}{\longrightarrow} \Z^r \arw 0$
and set $\pi(A)=\{u'_0,\ldots,u'_r\}$.
We identify $\ker \pi  $ with $  \Z^{n-r}$ by taking a basis of $\ker \pi$.
Then the isomorphism defined by
\begin{align}\label{eq_splittiing}
\Z^n \arw \ker \pi \times \Z^r =\Z^{n-r} \times \Z^r  \ :  \ u \mapsto ( u- s(\pi(u)), \pi(u) )
\end{align}
maps $A$ onto $ \bigcup_{i=0}^r (A_i \times \{u'_i\} )$,
where 
\[
A_i := (\pi^{-1}(u'_i) \cap A ) - s(u'_i) \subset \ker \pi = \Z^{n-r}
\]
is the parallel translation of $ \pi^{-1}(u'_i) \cap A $ by $- s(u'_i) $.

Let $g$ be the $\Z$-affine isomorphism $\Z^r \arw \Z^r$ defined by $g(u'_0)=0$ and $g(u'_i) =e_i $ for $1 \leq i \leq r$.
Let $f : \Z^ n \arw \Z^{n-r} \times \Z^r$ be the composite of \ref{eq_splittiing} and $\id_{\Z^{n-r}} \times g : \Z^{n-r} \times \Z^r \arw \Z^{n-r} \times \Z^r$.
By construction,
$f(A)= A_0 * \cdots * A_r$ and $\pr_2 \circ f = g \circ \pi$ hold.
(i) follows from $f(A)= A_0 * \cdots *A_r$ and (ii) follows from the commutativity $\pr_2 \circ f = g \circ \pi$.
\end{proof}

\begin{defn}\label{def_join_type_wrt}
A Cayley structure $\pi : \Z^n \arw \Z^{r}$ of $A$ is called 
\emph{of join type} if
the Cayley sum $A_0 * \cdots * A_r$ obtained in \autoref{lem_projection_cayley} is of join type.
This is equivalent to the condition that
the subspace $\sum_{i=0}^r M_i \subset \ker \pi$
is the inner direct sum of $M_i$ 's,
where $\pi(A)=\{u'_0,\ldots,u'_r\}$ and
\[
M_i := \langle \pi^{-1}(u'_i) \cap A  - \pi^{-1}(u'_i) \cap A  \rangle.
\]
We note that $M_i \subset \Z^n$ is contained in $\ker \pi$
even though $ \pi^{-1}(u'_i) \cap A $ is not contained in $\ker \pi$ in general.
\end{defn}

A partial order on the set of equivalence classes of Cayley structures of $A$
are introduced in \cite{IZ}:

\begin{defn}\label{def_order}
Let $\pi : \Z^n \arw \Z^r$ and $\pi' :\Z^n \arw \Z^{r'} $ be Cayley structures of $A $.
We define $\pi$ to be \emph{smaller than or equal} to $\pi'$ (written $\pi \preceq \pi'$)
if and only if $\pi$ factors through $\pi'$, i.e.,
there exists a surjective group homomorphism $ \varpi : \Z^{r'} \arw \Z^r$
such that $\pi =\varpi \circ \pi'$.
In other words,
$\pi \preceq \pi'$ if and only if $X_{\pi(A)} $ is contained in $X_{\pi'(A)}$ as subvarieties of $X_A$.

We note that both $\pi \preceq \pi'$ and $\pi' \preceq \pi$ hold if and only if $\pi$ and $\pi'$ are equivalent.
\end{defn}

Finally, we see one more lemma.
The simplest case of the lemma is the following example.

\begin{ex}\label{ex_r_r'}
Let $A'_0,\ldots,A'_r,A'_{r+1} \subset \Z^{n-r-1}$ be finite subsets for $0 \leq r \leq n-1$.
Set $A_0,\ldots,A_r \subset \Z^{n-r-1} \times \Z e_{r+1} \simeq \Z^{n-r}$ by
\[
A_0 = A'_0 \times \{0\} \cup A'_{r+1} \times \{e_{r+1}\} , \quad
A_i  = A'_i \times \{0\} \ \text{ for } \ 1 \leq i \leq r .
\]
Then $A_0 * \cdots * A_r \subset (\Z^{n-r-1} \times \Z e_{r+1}) \times \Z^r$ is $\Z$-affinely equivalent to $A'_0 * \cdots * A'_{r} * A'_{r+1} \subset \Z^{n-r-1} \times \Z^{r+1}$.
Since 
\begin{align*}
\langle A_0 -A_0 \rangle &= \langle A'_0 -A'_0 \rangle   \times \{0\} + \langle A'_{r+1} -A'_{r+1} \rangle  \times \{0\} + \Z (0,e_{r+1}), \\
\langle A_i -A_i \rangle &= \langle A'_i -A'_i \rangle   \times \{0\} \ \text{ for } \  1 \leq i \leq r 
\end{align*}
in $\Z^{n-r-1} \times \Z e_{r+1}$,
we have that
$A_0 * \cdots * A_r$
is of join type if so is $A'_0 * \cdots * A'_{r} * A'_{r+1}$.
\end{ex}

By repeating this argument, we have the following lemma.

\begin{lem}\label{lem_pi_pi'}
Let $\pi : \Z^n \arw \Z^r$ and $\pi' :\Z^n \arw \Z^{r'} $ be Cayley structures of $A $ with $\pi \preceq \pi'$.
If $\pi'$ is of join type,
so is $\pi$.
\end{lem}

\begin{proof}
By the induction of $r' -r$,
it suffices to show the case $r'=r+1$.
In this case,
this lemma follows from \autoref{ex_r_r'}.
\end{proof}


\section{Cayley structure induced from planes}\label{sec_refinement}

\subsection{Normal toric varieties defined by lattice polytopes}
In this section,
we consider normal toric varieties defined by polytopes.
Hence we recall some notations about normal toric varieties.
We refer the reader to \cite{Fu}, \cite{CLS} for a further treatment.

A \emph{lattice polytope} in $\R^n$ is the convex hull of a finite set in $\Z^n$.
The \emph{dimension} of a lattice polytope $P \subset \R^n$ is
the dimension of the affine space spanned by $P$.
For a subset $S$ in an $\R$-vector space,
we denote by $\Conv(S)$ (resp.\ $ \Cone(S)$) the convex hull of $S$
(resp.\ the closed convex cone spanned by $S$).
We say that subsets $S , S'$ in $ \R^n$ are $\Z$-\emph{affinely equivalent}
if there exists a $\Z$-affine automorphism of $\R^n$ which maps $S$ onto $S'$.

Let $P \subset \R^n$ be a lattice polytope of dimension $n$.
Then we can define the polarized toric variety associated to $P$ as
\[
(X_P,L_P)= (\Proj \C[\Gamma_P], \calo(1)),
\]
where $\Gamma_P:=\Cone(\{1\} \times P) \cap (\N \times \Z^n)$ is a subsemigroup of $\N \times \Z^n$.
We consider that $\Gamma_P$ is graded by $\N$,
that is,
the degree $k$ part of $\Gamma_P $ is $\Gamma_P \cap (\{k\} \times \Z^n)=\{k\} \times (kP \cap \Z^n) $. 
There exists a natural action on $X_P$ by the torus $(\*c)^n$.
We denote the maximal orbit in $X_P$ by $O_P = (\*c)^n.$
We regard $1_n \in (\*c)^n$ as a point in $X_P$ by $ (\*c)^n= O_P \subset X_P$.

By definition,
a lattice point $u \in P \cap \Z^n$ corresponds to a global section $x^u \in H^0(X_P,L_P)$.
It is well known that such global sections form a basis of $H^0(X_P,L_P)$
and the linear system $|L_P|$ is base point free.

\vspace{2mm}
Let $A \subset \Z^n$ be a finite set 
and assume that the lattice polytope $P :=\Conv(A) \subset \R^n$ is $n$-dimensional.
Then the linear system $| \bigoplus_{u \in A} \C x^u| \subset |L_P|$
defines a finite morphism $\nu : X_P \arw \P^{N}$ for $N = \# A -1$.
The image is nothing but $X_A \subset \P^N$.
If $\langle A - A \rangle =\Z^n$,
$\nu$ is the normalization of $X_A$.

\subsection{Refinement of a result in \cite{It}}


\begin{defn}\label{def of planes}
Let $(X,L)$ be a polarized variety
and let $r$ be a non-negative integer.
A subvariety $Z \subset X$ is called an $r$-\emph{plane}
if $(Z,L|_{Z}) $ is isomorphic to $ (\P^r, \calo_{\P^r} (1))$ as a polarized variety.
\end{defn}

We note that \autoref{def of planes} works even if $L$ is not very ample.
If $L=\calo_{\P^N}(1)|_X$ for an embedded projective variety $X \subset \P^N$,
a subvariety $Z \subset X$ is an $r$-plane if and only if $Z \subset \P^N$ is an $r$-dimensional linear projective subspace.

In this paper,
we consider $r$-planes in $(X_P,L_P)$ or $(X_A,\calo_{\P^N}(1)|_{X_A})$.
We say that a subvariety of $X_P$ or $X_A$ is an $r$-plane under these polarizations.

\vspace{2mm}
We can define Cayley structures of lattice polytopes similar to the case of $A \subset \Z^n$.

\begin{defn}\label{def_cayley_structure_polytope}
Let $P \subset \R^n$ be a lattice polytope of dimension $n$.
A surjective group homomorphism $\pi : \Z^n \arw \Z^r$ is called a \emph{Cayley structure} of $P$
if $\pi_{\R}(P)$ is $\Z$-affinely equivalent to $\Conv(0,e_1,\ldots,e_r)$,
where $\pi_{\R } : \R^n \arw \R^r$ is the induced $\R$-linear map.
\end{defn}

An equivalence relation and a partial order are defined by the same way as the case of $A \subset \Z^n$.
Cayley structures $\pi : \Z^n \arw \Z^r$ and $\pi' :\Z^n \arw \Z^r $ of $P$ are \emph{equivalent} if and only if 
there exists a group isomorphism $\iota : \Z^r \arw \Z^r$ such that $\pi = \iota \circ \pi'$.
For Cayley structures $\pi : \Z^n \arw \Z^r$ and $\pi' :\Z^n \arw \Z^{r'} $ of $P$,
$\pi$ is \emph{smaller than or equal} to $\pi'$ (written $\pi \preceq \pi'$)
if and only if
there exists a surjective group homomorphism $ \varpi : \Z^{r'} \arw \Z^r$
such that $\pi =\varpi \circ \pi'$.

Let $\pi : \Z^n \arw \Z^r$ be a Cayley structure of an $n$-dimensional lattice polytope $P \subset \R^n$. 
Then $\pi$ induces a surjective semigroup homomorphism $\Gamma_P \arw \Gamma_{\pi_{\R}(P)}$.
Hence we have an embedding $X_{\pi_{\R}(P)} \hookrightarrow X_P$ with $L_P |_{X_{\pi_{\R}(P)}} = L_{\pi_{\R}(P)} $.
Since $\pi_{\R}(P)$ is $\Z$-affinely equivalent to $\Conv(0,e_1,\ldots,e_r)$,
$X_{\pi(P)} \subset X_P$ is an $r$-plane.
The following result states that
conversely 
we can find such $\pi : \Z^n \arw \Z^r$ from an $r$-plane in $X_P$.

\begin{thm}[{\cite[Theorem 1.2]{It}}]\label{thm_CP}
Let $P \subset \R^n$ be a lattice polytope of dimension $n$
and let $r \leq n$ be a non-negative integer.
Then there exists a Cayley structure $\pi : \Z^n \arw \Z^{r}$ of $P$
if and only if 
there exists an $r$-plane $Z \subset X_P$ such that $1_n \in Z \cap O_P $.
\end{thm}

To prove \autoref{thm_structure},
we want a projection $\pi$ such that $X_{\pi_{\R}(P)}$ \emph{contains} $Z$
as we explained in Introduction.
However,
the above theorem does not give such $\pi$.
Roughly this is because we construct $\pi$ so that $Z$ degenerates to $X_{\pi_{\R}(P)}$ by choosing a monomial order of $\N^r$ in the proof,
although the degeneration does not appear explicitly.
In particular, the $r$-plane $X_{\pi_{\R}(P)}$ constructed in \cite{It} depends on the monomial order.
The following theorem gives a refinement of \autoref{thm_CP}.

\begin{thm}\label{thm_plane_ample}
Let $P \subset \R^n$ be a lattice polytope of dimension $n$.
Let $Z \subset X_P$ be an $l$-plane with $1_n \in Z \cap O_P$ for $l \geq 0$.
Then there exists a Cayley structure $\pi : \Z^n \arw \Z^{r}$ of $P$ for some $r \geq 0$ such that
\begin{itemize}
\item[(i)] the $r$-plane $X_{\pi_{\R}(P)} \subset X_P$ contains $Z $,
\item[(ii)] if a Cayley structure $\pi' : \Z^n \arw \Z^{r'}$ of $P$ also satisfies (i),
then $\pi \preceq \pi'$ holds.
\end{itemize}
\end{thm}

\begin{rem}\label{rem_unique}
By (i), we have $r \geq l$.
Hence \autoref{thm_CP} follows from \autoref{thm_plane_ample}.

By (ii),
such Cayley structure $\pi : \Z^n \arw \Z^{r}$ is uniquely defined by $P $ and $Z$ up to equivalence.
In other words,
the $r$-plane $X_{\pi_{\R}(P)} \subset X_P$ is uniquely determined by $P $ and $Z$.
\end{rem}

\begin{proof}[Proof of \autoref{thm_plane_ample}]

\textbf{Step 1.}
Without loss of generality,
we may assume that $0 \in \Z^n$ is a vertex of $P$. 
Let $x^u \in H^0(X_P,L_P)$ be the section corresponding to $u \in P \cap \Z^n$,
and let $D_u \in |L_P|$ be the corresponding divisor.
Since $1_n$ is contained in $Z$ and not contained in $D_u$,
we have $Z \nsubset  D_u$.
Hence $D_u |_Z$ is a divisor on $Z$.
Since $Z$ is an $l$-plane, 
$D_u |_Z$ is a hyperplane on $Z \simeq \P^{l}$.


Consider the following set of hyperplanes on $Z$
\[
\mathscr H = \big\{ H \subset Z \, \big| \, H = D_u |_Z \text{ for some } u \in P \cap \Z^n \big\}.
\]
Let $H_0,\ldots,H_r $ be all the members of $\mathscr H$ for $r= \# \mathscr H -1$.
Since $\bigcup_{u \in P \cap \Z^n} \Supp D_u = X_P \setminus O_P$,
we have $\bigcup_{i=0}^r H_i = Z \setminus (Z \cap O_P) $.

We can decompose $P \cap \Z^n $ into the disjoint union of $ A^0, A^1 \cdots, A^r$ by
\[
A^i := \big\{ u \in P \cap \Z^n \, \big| \, D_u |_Z = H_i \subset Z \big\}.
\]
We may assume that $0 \in P \cap \Z^n $ is contained in $ A^0$.
We note that each $A^i$ is not empty.
In the following steps,
we will construct $\pi : \Z^n \arw \Z^r$ which contracts each $A^i$ to a point.
\vspace{3mm}

\noindent
\textbf{Step 2.}
We define a graded semigroup homomorphism
\[
\beta : \Gamma_{P} = \Cone (\{1\} \times P) \cap (\N \times \Z^n) \arw \N \times  \N^r
\]
as follows.
For $(k,u) \in \Gamma_P$,
we denote by $x^{(k,u)} \in H^0(X_P, k L_P)$ the corresponding section.
Hence we have the corresponding divisor $D_{(k,u)} \in |k L_P|$ on $X_P$.
Since $\Supp D_{(k,u)} \subset X_P \setminus O_P$,
$\Supp D_{(k,u)} |_Z $ is contained in $ Z \setminus (Z \cap O_P) =\bigcup_{i=0}^r H_i $.
Thus we can write
\[
D_{(k,u)} |_Z = a_0 H_0 + \cdots + a_r H_r \in \big| k L_P |_Z \big| \simeq |\calo_{\P^l}(k)| 
\]
uniquely for some $a_i \in \N$ with $\sum_{i=0}^r a_i =k$.
We define 
\begin{align}\label{eq_def_beta}
\beta (k,u) =( k, a_1 , \ldots, a_r) \in \N \times \N^r.
\end{align}
We note that we do not take  $a_0$ in the definition of $\beta$.
For $(k,u) , (k',u')\in \Gamma_P$,
we have 
\[
x^{(k,u)} \cdot  x^{(k',u')} = x^{(k+k',u+u')} \in H^0(X_P, (k+k') L_P).
\]
Hence $D_{(k+k',u+u')} |_Z = D_{(k,u)} |_Z  +D_{(k',u')} |_Z $ holds.
Thus $\beta$ is a semigroup homomorphism.

Set $\Delta_r = \Conv (0,e_1,\ldots,e_r) \subset \R^r$.
Since $a_i \in \N$ and $a_1+ \cdots + a_r \leq k$ for $a_i$ in \ref{eq_def_beta},
we have $\beta(\Gamma_P) \subset \Gamma_{\Delta_r} = \Cone (\{1\} \times \Delta_r) \cap (\N \times \Z^r)$.

For $u \in P \cap \Z^n$,
the divisor $D_{(1,u)}$ is nothing but $D_u$ in Step 1.
Hence we have $D_{(1,u)} |_Z =H_i$ as divisors for $u \in A^i$.
Thus
\begin{align}\label{eq_beta(1,u)}
\beta (1,u)=
\left\{
\begin{array}{ll}
(1,0) &  \mbox{if \ $u \in A^0$}  \ \\
(1,e_i) &  \mbox{if \ $u \in A^i$ \ for \ $1 \leq i \leq r$}. \  
\end{array}
\right .
\end{align}
In particular,
we have $\beta (\Gamma_P) = \Gamma_{\Delta_r}$ since $ \Gamma_{\Delta_r} $ is generated by $(1,0) ,(1,e_1), \ldots, (1,e_r)$ as a semigroup. 
\vspace{3mm}

\noindent
\textbf{Step 3.}
Since $\beta$ is a semigroup homomorphism and $\Gamma_P \subset \N \times \Z^n$ spans $\Z \times \Z^n$ as a group,
$\beta $ uniquely extends to a group homomorphism $\Z \times \Z^n  \arw \Z \times \Z^r$.
We denote the group homomorphism by the same letter $\beta$.
We define a $\Z$-affine homomorphism $\pi : \Z^n \arw \Z^r$ by 
\[
\pi : \Z^n \simeq \{1\} \times \Z^n \subset \Z \times \Z^n \stackrel{\beta}{\longrightarrow} \Z \times \Z^r \arw \Z^r,
\]
where the last homomorphism is the projection to the second factor.
Then
\begin{align}\label{eq_pi(u))}
\pi(u)=
\left\{
\begin{array}{ll}
0 &  \mbox{if \ $u \in A^0$}  \ \\
e_i &  \mbox{if \ $u \in A^i$ \ for \ $1 \leq i \leq r$} \  
\end{array}
\right .
\end{align}
by \ref{eq_beta(1,u)}.
In particular,
$\pi(0)=0 \in \Z^r$ holds since $0 \in A^0 $.
Hence $\pi$ is a surjective group homomorphism. 
Since $\pi_{\R} (P) = \Delta_{r}$ by \ref{eq_pi(u))},
$\pi$ is a Cayley structure of $P$.
The embedding $X_{\pi_{\R}(P)} \subset X_P$ is induced by the graded $\C$-algebra homomorphism 
\[
\Psi : \C [\Gamma_P] \arw \C[\Gamma_{\Delta_r}] \ \  : \ \ x^{(k,u)} \mapsto x^{\beta(k,u)}.
\]

\vspace{2mm}

\noindent
\textbf{Step 4.}
To show (i),
it suffices to see that 
the natural homomorphism
\begin{align}\label{eq_section_ring}
\C [\Gamma_P] = \bigoplus_{k \in \N} H^0(X_P, k L_P) \arw  \bigoplus_{k \in \N} H^0(Z, k L_P |_Z)
\end{align}
factors through $\Psi$.
Since $\Psi$ is induced by $\beta$,
the kernel of $ \Psi$ is generated by 
\[
\{ x^{(k,u)} -  x^{(k,u')} \, | \,  (k,u), (k,u') \in \Gamma_P, \beta(k,u) = \beta(k,u') \} .
\]
By the following claim,
$\ker \Psi$ is contained in the kernel of \ref{eq_section_ring},
which implies $Z \subset X_{\pi_{\R}(P)}$.

\begin{claim}\label{claim_kernel}
If $\beta(k,u) = \beta(k,u')$ for $(k,u), (k,u') \in \Gamma_P$,
it holds that $x^{(k,u)} |_Z = x^{(k,u')} |_Z  \in H^0(Z, k L_P |_Z) \simeq H^0(\P^l,\calo(k))$.
\end{claim}

\begin{proof}
For $(k,u) \in \Gamma_P$, we can write
\[
D_{(k,u)} |_Z = a_0 H_0 + \cdots + a_r H_r \in \big| L_P |_Z \big| \simeq |\calo_{\P^l}(k)| 
\]
for some $a_i \in \N$ with $\sum_{i=0}^r a_i =k$ by Step 2.
Hence $\beta(k,u) = \beta(k,u')$ holds if and only if
$D_{(k,u)} |_Z  = D_{(k,u')} |_Z $ holds as divisors on $Z$.
Thus we have $x^{(k,u)} |_Z = a \cdot x^{(k,u')} |_Z $ for some $a \in \*c$.
By substituting $1_n \in Z \cap O_P$ to this equality,
we have $a =1 $.
\end{proof}
\vspace{3mm}

\noindent
\textbf{Step 5.}
To show (ii),
take another Cayley structure $\pi' : \Z^n \arw \Z^{r'}$ of $P$ such that
 $Z \subset X_{\pi'_{\R}(P)} \subset X_P$.

Since $\pi'_{\R}(P)$ is $\Z$-affinely equivalent to $\Delta_{r'}$,
$(\ker \pi')_{\R}=\ker (\pi'_{\R})$ is generated by 
\[
\{ u-u' \, | \, u,u' \in P \cap \Z^n, \pi'(u)= \pi'(u') \}
\]
as an $\R$-vector space.
If $\pi'(u)= \pi'(u') $ for $u,u' \in P \cap \Z^n$,
we have 
\[
x^u |_{ X_{\pi'_{\R}(P)} } = x^{u'} |_{ X_{\pi'_{\R}(P)} } \in H^0( X_{\pi'_{\R}(P)} , L_P |_{ X_{\pi'_{\R}(P)} }) = H^0( X_{\pi'_{\R}(P)} , L_{ \pi'_{\R}(P)} ).
\]
Since $Z \subset X_{\pi'_{\R}(P)}$,
$x^u |_Z = x^{u'} |_Z \in H^0( Z , L_P |_Z)$ holds,
and hence $ D_u |_Z = D_{u'} |_Z$.
By the definition of $A^i$'s and \ref{eq_pi(u))},
we have $\pi(u) = \pi(u')$,
i.e.,
$u-u' \in \ker \pi$.
Hence $\ker \pi' \subset \ker \pi$,
which implies that $\pi$ factors through $\pi'$.
\end{proof}

\begin{lem}\label{P-rem_minimality}
Let $\pi : \Z^n \arw \Z^r$ be the Cayley structure of $P$ in the statement of \autoref{thm_plane_ample}
and let $\overline{H}_0, \overline{H}_1, \dots, \overline{H}_r \subset X_{\pi_{\R}(P)} \simeq \P^r$
be the torus invariant hyperplanes.
Then we have $\overline{H}_i |_Z \neq \overline{H}_j |_Z$ for any $i \neq j$ as hyperplanes on $Z$.
\end{lem}
\begin{proof}
  We use the notation in the proof of \autoref{thm_plane_ample}.
  By renumbering the indexes,
  we may assume that $\overline{H}_i$ is the hyperplane corresponding to $\pi(A^i) \in \pi_{\R}(P) \cap \Z^r = \{0,e_1,\dots,e_r\} $.
Then we have
  $\overline{H}_i  = D_u |_{X_{\pi_{\R}(P)}}$ for $u \in A^i$.
  Since $\overline{H}_i |_Z = D_u |_Z =H_i $ for $u \in A^i$,
  the assertion follows.
\end{proof}

As a corollary of \autoref{thm_plane_ample},
we have a similar statement for $X_A \subset \P^N$.

\begin{cor}\label{thm_plane_embedding}
Let $A \subset \Z^n$ be a finite subset with $\langle A -A \rangle=\Z^n$.
Let $Z \subset X_A $ be an $l$-plane with $1_{n} \in Z $ for $l \geq 0$.
Then there exists a Cayley structure $\pi : \Z^n \arw \Z^{r}$ of $A$ for some $r \geq 0$ such that
\begin{itemize}
\item[(i)] the $r$-plane $X_{\pi(A)} \subset X_A$ contains $Z $,
\item[(ii)] if a Cayley structure $\pi' : \Z^n \arw \Z^{r'}$ of $A$ also satisfies (i),
then $\pi \preceq \pi'$ holds.
\end{itemize}
\end{cor}

\begin{proof}
Set $P = \Conv(A) \subset \R^n$.
Then we have the normalization $\nu : X_P \arw X_A$ such that $\nu^* \calo_{X_A}(1) =L_P$.
Let $Z' \subset X_P$ be the strict transform of $Z \subset X_A$.
Since $\nu$ is an isomorphism over $(\*c)^n \subset X_A$ and $Z \cap (\*c)^n \neq \emptyset$,
$\nu |_{Z'} : Z' \arw Z$ is birational and finite.
Since $Z \simeq \P^l$ is normal, $\nu |_{Z'} $ is an isomorphism by Zariski's main theorem.
Hence $1_n \in Z' \subset X_P$ is an $l$-plane.
Thus we can apply \autoref{thm_plane_ample} to $P, Z'$
and obtain $\pi : \Z^n \arw \Z^r$,
which is a Cayley structure of $A$ and satisfies (i), (ii).
\end{proof}

\begin{rem}\label{rem_IZ}
In \cite[Section 5]{IZ},
Ilten and Zotine prove \autoref{thm_plane_embedding} independently,
as a key step to describe the Fano schemes of $X_A$.
We note that they also consider $l$-planes in $X_{A}\setminus (\*c)^n$.
\end{rem}

We also have the following lemma in a similar way to \autoref{P-rem_minimality}.
This will be used in the proof of \autoref{thm_structure}.

\begin{lem}\label{rem_minimality}
  Let $\pi : \Z^n \arw \Z^r$ be the Cayley structure of $A$ in the statement of \autoref{thm_plane_embedding}
and let $\overline{H}_0, \overline{H}_1, \dots, \overline{H}_r \subset X_{\pi(A)} \simeq \P^r$
be the torus invariant hyperplanes.
Then we have $\overline{H}_i |_Z \neq \overline{H}_j |_Z$ for any $i \neq j$ as hyperplanes on $Z$.
\end{lem}

\begin{ex}\label{ex_line_in_P^2}
Let $A=\{0,e_1,e_2\} \subset \Z^2$.
Then we have homogeneous coordinates $X_0,X_1,X_2 \in H^0(X_A,\calo_{X_A}(1))$ on $X_A=\P^2$ 
corresponding to $0,e_1,e_2$.
In this case,
the $1$-plane $X_{\pi(A)} \subset X_A$
induced from some Cayley structure $\pi :\Z^2 \arw \Z^1$ of $A$
is one of the following:
\begin{align} \label{eq_3lines}
 \Zero(X_0-X_1), \quad \Zero(X_0-X_2), \quad  \Zero(X_1-X_2)  \quad \subset \quad  \P^2.
\end{align}

Consider $1$-planes $Z_1,Z_2$ in $ X_A=\P^2$ containing $[1:1:1] $ defined by
\[
Z_1=\Zero(X_1-X_2), \quad Z_2= \Zero (2X_0 - X_1-X_2).
\]
Since $Z_1=\Zero(X_1-X_2)$ coincides with the $1$-plane $X_{\pi(A)} \subset X_A$
induced by the Cayley structure
\[
\pi : \Z^2 \arw \Z^1 \ \ : \ \  e_1 \mapsto 1 , e_2 \mapsto 1,
\]
this $\pi$ satisfies (i), (ii) in \autoref{thm_plane_embedding} for $Z_1$.

On the other hand,
$Z_2$ does not coincide with any $1$-plane in \ref{eq_3lines}. 
Hence 
$\id_{\Z^2} : \Z^2 \arw \Z^2$ satisfies (i), (ii) in \autoref{thm_plane_embedding} for $Z_2$.
\end{ex}

\section{Cayley sum and Contact locus}\label{sec_contact_locus}

\begin{defn}\label{def_Z_H}
Let $A \subset \ZZ^n$ be a finite subset with $\langle A-A\rangle=\Z^n$ and let $H \subset \P^N$ be a hyperplane.
We set
\[
Z_{A,H}^{\circ} := \{ x \in (\*c)^n  \, | \,  H  \text{ is tangent to } X_A  \text{ at } x \}  \subset (\*c)^n \subset  X_A
\]
and let $Z_{A,H} \subset X_A$ be the closure of $Z_{A,H}^{\circ}$.
We call $ Z_{A,H} $ the \emph{contact locus} of $H$ on $X_A$.
\end{defn}

It is known that
$Z_{A,H} \subset X_A \subset \P^N$ is a $\delta_{X_A} $-plane 
if $H \in (\P^{N})^{\vee}$ is a general point of $X_A^*$ (see \cite[Theorem 1.18]{Te} for example).
Because of the torus action on $X_A$,
$Z_{A,H} $ is a $\delta_{X_A} $-plane if $H$ is a general hyperplane which is tangent to $X_A$ at $1_n$.

\vspace{4mm}
In this section,
we fix finite sets $A_0,\ldots,A_r \subset \Z^{n-r} $
and let $A$ be the Cayley sum $A_0 * \cdots * A_r \subset \Z^{n-r} \times \Z^r$. We assume $\langle A-A \rangle =\Z^{n-r} \times \Z^r$.

By identifying  $(\*c)^r $ with $\{1_{n-r}\} \times (\*c)^r $,
we regard $ (\*c)^r $ as a subtorus of $ (\*c)^{n-r} \times (\*c)^r $.
Then the closure of $ (\*c)^r  $ in $ X_A$
is nothing but the $r$-plane $X_{\pi(A)} \subset X_A$
induced by the projection $\pi : \Z^{n-r} \times \Z^{r} \arw \Z^r$ to the second factor.
Our aim is to describe $Z_{A,H}^{\circ}  \cap (\*c)^r$.

\vspace{2mm}

Write $A_i= \{u_{i0}, \ldots,u_{iN_i}\} \subset \Z^{n-r}$ for $ N_i= \# A_i -1$,
and set $N=\# A -1 = \sum_{i=0}^r (N_i+1) -1$.
Let $\{ X_{ij}\}_{0 \leq i \leq r, 0 \leq j \leq N_i}$ be the homogeneous coordinates on $ \PP^{N}$,
and let
$z_1,\ldots,z_{n-r},w_1,\ldots,w_r $ be the coordinates of $ (\*c)^{n-r} \times  (\*c)^r $.
Then the embedding $\phi_A :  (\*c)^{n-r} \times  (\*c)^r \hookrightarrow \PN$
is defined by $(z_1,\ldots,z_{n-r},w_1,\ldots,w_r) \mapsto [w_iz^{u_{ij}}]_{0 \leq i \leq r, 0 \leq j \leq N_i}$,
where we set $w_0 = 1 \in \CC$.
Let us take a linear form
\[
\sum_{0 \leq i \leq r, 0 \leq j \leq N_i} a_{ij} X_{ij}
\ \ \text{ with }\ \, a_{ij} \in \C,
\]
and consider the hyperplane $H \subset \PN$ defined by this linear form.

The following lemma 
is well known. See
\cite[Proposition 4.1]{DFS} for the proof.


\begin{lem}\label{thm:tan-1-w}
  Let $H$ be as above and take $w \in   (\*c)^r$. Then $H$ is tangent to $X_A $
  at $(1_{n-r},w) \in  (\*c)^{n-r} \times  (\*c)^r
  $ 
  if and only if
  $(a_{ij})_{i,j} \in \C^{N+1} $
  satisfies the condition
  \begin{equation}\label{eq_condition}
    \begin{aligned}
      \sum_{0 \leq j \leq N_i} a_{ij}    &=0 \in \C \quad \text{for} \quad  0 \leq i \leq r,\\
      \sum_{0 \leq i \leq r} w_i \sum_{0 \leq j \leq N_i} a_{ij}    u_{ij} &=0 \in (\Z^{n-r})_{\C}= \C^{n-r}.
    \end{aligned}
  \end{equation}  
In particular,
  $H$ is tangent to $X_A$ at $(1_{n-r},1_r) $ if and only if 
  \begin{equation}\label{eq:cond-tan-1-1}
\sum_{0 \leq j \leq N_i} a_{ij}    =0 \in \C \,\ \text{ for }\ 0 \leq i \leq r, \quad \sum_{0 \leq i \leq r}  \sum_{0 \leq j \leq N_i} a_{ij}    u_{ij} =0 \in \C^{n-r}.   
  \end{equation}
\end{lem}

\vspace{4mm}
Now we assume that the hyperplane
$H$ is tangent to $X_A$ at $(1_{n-r},1_r)$,
i.e., $(a_{ij})_{i,j}$ satisfies \ref{eq:cond-tan-1-1}.
We set
\[
m_i = \sum_{0 \leq j \leq N_i} a_{ij}u_{ij} \in \CC^{n-r}
\]
and consider
the linear map
\[
\xi : \C^{r+1} \arw \C^{n-r} \quad : \quad
 (W_0,\ldots,W_r ) \mapsto \sum_{0 \leq i \leq r} W_i m_i .
\]
Note that
for $\P^r = \P_{\! *} (\C^{r+1})
  =  (\C^{r+1} \setminus \{0\}) / \*c$,
the torus $(\C^{\times})^{r} $ is embedded in $X_{\pi(A)} = \P^r$
as an open dense subset by $(w_1, \dots, w_r) \mapsto [1:w_1:\cdots:w_r]$.

We give a description of the intersection of the torus $(\C^{\times})^{r}$
and the open subset of the contact locus $Z_{A,H}^{\circ} = (\C^{\times})^{n} \cap Z_{A,H}$ defined in \autoref{def_Z_H}.
\begin{lem}\label{lem_tangent_dim}
  Assume that $H$ is tangent to $X_A$ at $(1_{n-r},1_r) $.
  Then
  \begin{enumerate}
  \item $Z_{A,H}^{\circ} \cap (\*c)^r =\{ w = (w_1, \dots, w_r) \in (\C^{\times})^{r} \, | \,  \sum_{0 \leq i \leq r } w_i m_i = 0 \in \C^{n-r} \}$
    and the closure is equal to
    $\P_{\! *} (\ker \xi)
    \subset \P_{\! *} (\C^{r+1}) =X_{\pi(A)} \subset X_A$,
    where $w_0 := 1$.
  \item $\dim Z_{A,H}^{\circ} \cap (\*c)^r  = r -  \dim \langle m_0, \ldots, m_r  \rangle_{\C}$,
    where $\langle m_0, \ldots, m_r  \rangle_{\C} \subset \C^{n-r} $ is the subspace spanned by $m_i$'s.
  \end{enumerate}
\end{lem}

\begin{proof}
  (1) The first equality follows from \ref{eq_condition}
  since $\sum_j a_{ij} = 0$ already follows from \ref{eq:cond-tan-1-1}.
  By the definition of $\xi$, we have 
  \begin{align}\label{eq_P_cap_T}
    \Big\{ w \in (\C^{\times})^{r} \, \big| \,  \sum_{0 \leq i \leq r } w_i m_i = 0 \in \C^{n-r} \Big\} = \P_{\! *} (\ker \xi) \cap  (\C^{\times})^{r}.
  \end{align}   
  Since $1_r \in  (\C^{\times})^{r} $ is contained in the left hand side of \ref{eq_P_cap_T} by assumption,
  $ \P_{\! *} (\ker \xi) \cap  (\C^{\times})^{r}$ is a non-empty open subset of $\P_{\! *} (\ker \xi)$,
  which is dense in $\P_{\! *} (\ker \xi)$.

\vspace{2mm}

\noindent
  (2)
  By the definition of $\xi$,  $\rank \xi  = \dim \langle m_0, \ldots, m_r  \rangle_{\C}$ holds. 
  Hence we have
  \begin{align*}
    \dim Z_{A,H}^{\circ} \cap (\*c)^r  &= \dim  \P_{\! *} (\ker \xi) \\
    &= \dim \ker \xi -1 \\
    &= r+1 - \rank \xi -1 = r -  \dim \langle m_0, \ldots, m_r  \rangle_{\C}.
  \end{align*}
  by (1).
\end{proof}

In order to clarify the description of the dimension of
$Z_{A,H}^{\circ} \cap (\*c)^r$ for \emph{general} $H$,
we introduce an invariant $\alpha$ 
as follows.

\begin{defn}\label{def_alpha2}
Let $V$ be a vector space of finite dimension over a field $k=\Q$ or $\C$,
and let $V_0,\ldots,V_r \subset V$ be subspaces.
We set $K$ to be the kernel of the linear map
\[
V_0 \oplus \cdots \oplus V_r \arw V \  :  \  (m_0,\ldots,m_r) \mapsto m_0 + \cdots + m_r,
\]
and define
\[
\alpha = \alpha(V_0,\ldots,V_r) := \max \left\{   \dim_k \langle m_0, \ldots, m_r  \rangle_k \, | \, (m_0,\ldots,m_r) \in K \right\} \leq r,
\]
where $\langle m_0, \ldots, m_r  \rangle_k \subset V$ is the subspace spanned by $m_0,\ldots,m_r \in V$.
Since  $\dim \langle m_0, \ldots, m_r  \rangle_k$ is a lower semicontinuous function on $K$ in the Zariski topology,
$\alpha =  \dim_k \langle m_0, \ldots, m_r  \rangle_k $ holds for {\it general}  $(m_0,\ldots,m_r) \in K$.
\end{defn}

\begin{rem}\label{rem_alpha_R&C}
Let $V$ be a $\Q$-vector space of finite dimension.
For subspaces $V_0,\ldots,V_r $ of $ V$,
let $(V_0)_{\C} ,\ldots, (V_r)_{\C} \subset V_{\C}$ be the $\C$-vector spaces obtained by tensoring $\C$ over $\Q$.
Then $\alpha(V_0,\ldots,V_r) = \alpha ((V_0)_{\C} ,\ldots, (V_r)_{\C} ) $ holds.
In other words,
$\alpha$ does not change by tensoring $\C$.
This is because the kernel of 
\[
(V_0 \oplus \dots \oplus V_r)_{\C}=(V_0)_{\C} \oplus \cdots \oplus (V_r)_{\C}  \arw V_{\C} \ : \ (m_0,\ldots,m_r) \mapsto m_0 + \dots + m_r
\]
coincides with $K_{\C}$, and $K $ is dense in $ K_{\C}$ in the Zariski topology.
\end{rem}

Let us return to the original setting.
We set
\[
X_{A, (1_{n-r},1_{r})}^* =\left\{ H \in \Pv \, | \, H \text{ is tangent to } X_A \text{ at } (1_{n-r},1_{r}) \right\}  \subset \Pv,
\]
which is a linear subvariety of codimension $n+1$ defined by \ref{eq:cond-tan-1-1}.
\begin{prop}\label{lem_tangent_dim_alpha}
  Set $V_i = \langle A_i-A_i \rangle_{\Q}$ in $V:=\Q^{n-r}$,
  and let $K \subset V_0 \oplus \cdots \oplus V_r$ be the subspace as in \autoref{def_alpha2}.
  Let $H \subset \P^N$ be the hypersurface defined by $\sum_{i,j} a_{ij} X_{ij} =0$, and assume that $H \in X_{A, (1_{n-r},1_{r})}^*$.
  Set $m_i = \sum_{0 \leq j \leq N_i} a_{ij}u_{ij} \in \CC^{n-r}$.
  Then the following holds.
  \begin{itemize}
  \item[(i)] $m_i \in (V_i)_{\C}$ and $(m_0,\ldots,m_r) $ is contained in $ K_{\C}$.
  \item[(ii)] $(m_0,\ldots,m_r)$ is general in $K_{\C}$ if $H$ is general  in $X_{A, (1_{n-r},1_{r})}^*$.
  \item[(iii)] $\dim Z_{A,H}^{\circ} \cap (\*c)^r = r -  \alpha (V_0\ldots,V_r)$ 
    for general $H \in X_{A, (1_{n-r},1_{r})}^*$. 
    In particular,
    we have $\delta_{X_A} \geq r -  \alpha (V_0\ldots,V_r)$.
    \item[(iv)] If there exists a surjective group homomorphism $ p : \Z^{n-r} \arw \Z^{n-r-c}$ such that $p(A_0) * \cdots * p(A_r) \subset \Z^r \times \Z^{n-r-c}$ is of join type,
      then we have $\alpha (V_0\ldots,V_r) \leq c$,
    and hence $\delta_{X_A} \geq r-c$.
  \end{itemize}
\end{prop}

\begin{proof}
  (i) Since $(a_{ij})_{i,j}$ satisfies the condition~\ref{eq:cond-tan-1-1},
  the statement follows.
\vspace{2mm}

\noindent
  (ii) Set
  $L = \set*{(a_{ij})_{i,j} \in \CC^{N+1}}{\text{$(a_{ij})_{i,j}$ satisfies \ref{eq:cond-tan-1-1}}}$, which is a linear subspace of $\CC^{N+1}$.
  Then $H$ is general in $X_{A,(1_{n-r},1_{r})}^*$
  if and only if $(a_{ij})_{i,j}$ is general in $L$.
  Consider the linear map
  \begin{align}\label{eq_K_to_K}
    L \arw K_{\C} : (a_{ij})_{i,j} \mapsto ( m_0,\ldots,m_r)
  \end{align}
  defined by $m_i = \sum_{0 \leq j \leq N_i} a_{ij}    u_{ij} $.

  Since any element of $(V_i)_{\C} =  \langle A_i-A_i \rangle_{\C}$ is written as 
  $ \sum_{0 \leq j \leq N_i} a_{ij}    u_{ij} $ with $ \sum_{0 \leq j \leq N_i} a_{ij}    =0 $, 
  \ref{eq_K_to_K} is surjective.
  Hence $(m_0,\ldots,m_r) \in K_{\C}$ is general for general $(a_{ij})_{i,j} \in L$.
\vspace{2mm}

\noindent
(iii)
  By (ii),
  it holds that $\dim \langle m_0, \ldots, m_r  \rangle_{\C} = \alpha ((V_0)_{\C},\ldots,(V_r)_{\C}) = \alpha (V_0\ldots,V_r)$.
  Hence $\dim Z_{A,H}^{\circ} \cap (\*c)^r  = r -  \alpha (V_0\ldots,V_r)$ holds
  by \autoref{lem_tangent_dim}. 
  The last statement follows from $\delta_{X_A} = \dim Z_{A,H}^{\circ} $ for general $H$.
  
\vspace{2mm}

\noindent
  (iv) 
  Let $ p_{\Q} : \Q^{n-r} \arw \Q^{n-r-c}$ be the surjective linear map induced by $p$.
  Since $p(A_0) * \cdots * p(A_r) \subset \Z^r \times \Z^{n-r-c}$ is of join type,
\begin{align}\label{eq_inner_direct_sum}
\langle p(A_0) -p(A_0)  \rangle_{\Q} + \cdots + \langle p(A_r) -p(A_r)  \rangle_{\Q} \subset \Q^{n-r-c}
\end{align}
is the inner direct sum of $\langle p(A_0) -p(A_0)  \rangle_{\Q} ,\ldots,  \langle p(A_r) -p(A_r)  \rangle_{\Q}$.

Take $(v_0,\ldots,v_r) \in K$.
Then $p_{\Q} (v_i )$ is contained in $ p_{\Q} (V_i) = \langle p(A_i) -p(A_i)  \rangle_{\Q}$.
Furthermore,
it follows from $\sum_{i=0}^r v_i  =0$ that $\sum_{i=0}^r  p_{\Q} (v_i ) = 0 \in \Q^{n-r-c}$.
Since \ref{eq_inner_direct_sum} is the inner direct sum,
we have $ p_{\Q} (v_i ) =0$ for any $0 \leq i \leq r$.
Hence $v_i \in \ker p_{\Q}$ and 
\[
\dim \langle v_0,\ldots,v_r \rangle_{\Q} \subset \ker p_{\Q} \simeq \Q^c,
\]
which implies $\alpha(V_0,\ldots,V_r) \leq c$.
\end{proof}

\begin{ex}
Let $A = A_0*\cdots* A_r \subset \Z^{n-r} \times \Z^r$ be a Cayley sum as above.
It is known that $X_A $ is dual defective if $2r-n >0$ (see \cite[Section 7]{Pi} for example).
This can be checked by \autoref{lem_tangent_dim_alpha} (iv) for $c=n-r$ and the zero map $p : \Z^{n-r} \arw \Z^0= \{0\}$
since $r-c =2r -n$ and 
\[
p(A_0) * \cdots * p(A_r)=\{0\} \times \{0,e_1,\dots,e_r\} \subset \{0\} \times \Z^{r}
\]
is of join type.
\end{ex}

We will use the following lemma about the invariant $\alpha$ in the next section.
In the lemma, we consider the following condition on $K$ in \autoref{def_alpha2}:
\begin{itemize}
\item[$(*)$ ] For general $(m_0,\ldots,m_r)  \in K$ and
for any two integers $i,j$ with $0 \leq i, j \leq r$, it holds that
\[
 \langle m_l \, | \, 0 \leq  l \leq r, l \neq i, j \rangle_k = \langle m_0,\ldots,m_r\rangle_k .
\]
\end{itemize}
In other words, for general $(m_0,\ldots,m_r)  \in K$,
even if any two of $m_0, \dots, m_r$ are removed,
the rest $m_l$'s still span the original space $\lin{m_0, \dots, m_r}_k$.

\begin{lem}\label{key_lem_linear_alg}
Let $V_0,\ldots,V_r $ and $K$ be as in \autoref{def_alpha2} and assume that the condition $(*)$ is satisfied.
Then there exists an $ \alpha(V_0,\ldots,V_r )$-dimensional subspace $V' \subset V$ such that
\begin{itemize}
\item[(1)] for the canonical quotient map $q : V \arw V/ V'$,
the subspace $q(V_0) + \dots + q(V_r) \subset V/V'$ is the inner direct sum of $q(V_0) , \dots , q(V_r) $,
\item[(2)] if $\tilde{q} : V \arw \tilde{V}$ is a linear map and $\tilde{q} (V_0) + \dots + \tilde{q} (V_r) \subset \tilde{V}$ is the inner direct sum of $\tilde{q} (V_i) $'s,
it holds that $V' \subset \ker \tilde{q}$.
\end{itemize}
\end{lem}

\begin{proof}
Set $\alpha = \alpha(V_0,\ldots,V_r ) $.
Fix a general element $(m_0,\ldots,m_r)  \in K$
and set $V' = \langle m_0,\ldots,m_r\rangle_k \subset V$.
By the definition of $\alpha(V_0,\ldots,V_r ) $,
we have $\dim V' = \alpha$.

\begin{claim}\label{claim_independence_W}
For any $(m'_0,\ldots,m'_r) \in K $,
each $m'_i $ is contained in $ V'$.
\end{claim}

\begin{proof}[Proof of \autoref{claim_independence_W}]
It is enough to show that $m'_0 \in V'$.

For any $1 \leq i \leq r$, taking $j=0$ in $(*)$, we have
\[
\lin{m_l \mid 1 \leq l \leq r, l \neq i}_k = \lin{m_0, \dots, m_r}_k.
\]
In particular,
$m_i  \in \lin{m_l \mid 1 \leq l \leq r, l \neq i}_k$,
i.e.,
$m_i$ is written as a linear combination of $m_l$'s
with $l \neq 0, i$.
Hence we can find $b_{i1}, \dots, b_{ir} \in k$ 
with $b_{ii} = 1$ such that
$\sum_{1 \leq l \leq r} b_{il} m_l =0$.

Fix general $(t_1,\ldots,t_r) \in k^r$ and set $\tilde{m}_l = (\sum_{1 \leq i \leq r} t_i b_{il} ) m_l \in V_l$ for $1 \leq l \leq r$.
Since
\[
\sum_{1 \leq l \leq r} \tilde{m}_l = \sum_{1 \leq l \leq r} \left(\sum_{1 \leq i \leq r} t_i b_{il} \right) m_l = \sum_{1 \leq i \leq r} t_i \sum_{1 \leq l \leq r} b_{il} m_l =0,
\]
we have $(0,\tilde{m}_1,\ldots,\tilde{m}_r) \in K$.
Since $b_{ii}=1 \neq 0$ and $(t_1,\ldots,t_r)$ is general,
$\sum_{1 \leq i \leq r} t_i b_{il} \neq 0$ holds for any $l$.
Hence we have
$ \langle \tilde{m}_1,\ldots,\tilde{m}_r\rangle_k = \langle m_1,\ldots,m_r\rangle_k =V'$,
where the second equality holds
since $m_0 = - \sum_{1 \leq i\leq r} m_i$ by the definition of $K$.

Since $(0,\tilde{m}_1,\ldots,\tilde{m}_r), (m'_0,m_1',\ldots,m'_r) \in K$,
we have
\begin{equation}\label{eq:s-sum-in-K}
  (s m'_0 , \tilde{m}_1 + s m'_1,\ldots,\tilde{m}_r + sm'_r) \in K
\end{equation}
for any $s \in k$.
Let us consider
\[
h(s) := \dim \langle m'_0, \tilde{m}_1 + s m'_1,\ldots,\tilde{m}_r + s m'_r\rangle_k,
\]
which is equal to
the dimension of
$\langle s m'_0, \tilde{m}_1 + s m'_1,\ldots,\tilde{m}_r + s m'_r \rangle_k$
if $s \neq 0$.
Then $h(s) \leq \alpha$ holds for any $s \neq 0$ because of \ref{eq:s-sum-in-K}
and the definition of $\alpha$.
Since $h(s)$ is lower semicontinuous as a function of $s \in k$,
we have $h(0) \leq \alpha$,
that is,
the dimension of 
$\langle m'_0, \tilde{m}_1,\ldots,\tilde{m}_r \rangle_k$ is at most $\alpha$ as well.

On the other hand,
the dimension of $ \langle \tilde{m}_1,\ldots,\tilde{m}_r \rangle_k  =V'$ is $\alpha$.
Hence $m'_0$ must be contained in $ \langle \tilde{m}_1,\ldots,\tilde{m}_r \rangle_k  =V'$.
\end{proof}

We show that this $V'$ satisfies the conditions (1), (2) in this lemma.

To see (1),
take $v_i \in V_i $ and assume that $q(v_0) + \dots + q(v_r) =0 \in V/V'$.
Then $v_0 + \cdots + v_r \in V' = \langle m_0,\ldots,m_r \rangle_k$.
Hence $v_0 + \cdots + v_r = \sum_{0 \leq i \leq r} c_i m_i $ holds for some $c_i \in k$.
This means $(v_0 - c_0 m_0 ,\ldots, v_r - c_r m_r ) \in K$.
By \autoref{claim_independence_W},
$v_i - c_i m_i \in V'$ holds.
Since $m_i \in V'$,
we have $v_i \in V'$ and hence $q(v_i) =0 \in V/V'$.
Thus $q(V_0) + \dots + q(V_r) \subset V/V'$ is the inner direct sum of $q(V_i)$'s.

To check (2),
let $\tilde{q}$ be a linear map as in (2).
Since $\sum_{0 \leq i \leq r} m_i =0 \in V$,
it holds that $\sum_{0 \leq i \leq r} \tilde{q} (m_i) =0 \in \tilde{V}$.
Hence we have $\tilde{q} (m_i) =0$ for any $i$
since $ \tilde{q} (m_i)  \in  \tilde{q} (V_i) $ and
$\tilde{q} (V_0) + \dots + \tilde{q} (V_r) \subset \tilde{V}$ is the inner direct sum.
Thus $m_i \in \ker \tilde{q}$ and $V' = \langle m_0,\ldots,m_r\rangle_k $ is contained in $ \ker \tilde{q}$.
\end{proof}

\section{Proof of \autoref{thm_structure}}\label{sec_proof}

\subsection{Paraphrase of \autoref{thm_structure} by Cayley structures}

To prove the main result \autoref{thm_structure},
we paraphrase 
 \autoref{thm_structure} by Cayley structures as follows so that we can apply results in \autoref{sec_refinement}:

\begin{thm}\label{refined_thm_structure}
 Let $A \subset \Z^n$ be a finite subset with $\langle  A - A  \rangle = \Z^n  $. 
 Then there exist surjective group homomorphisms $\pi_1 : \Z^n \arw \Z^{n-c}$ and $\pi_2 : \Z^{n-c} \arw \Z^r$
 for some non-negative integers $r,c$ such that
  \begin{itemize}
  \item[(1)] $\pi_2 \circ \pi_1  : \Z^n \arw \Z^r$ is a Cayley structure of $A$,
  \item[(2)] $\pi_2$ is a Cayley structure of $\pi_1(A) \subset \Z^{n-c} $ of join type, 
  \item[(3)] $\delta_{X_A} = r-c$,
  \item[(4)] if surjective group homomorphisms $\pi_1' : \Z^n \arw \Z^{n-c'}$ and $\pi_2' : \Z^{n-c'} \arw \Z^{r'}$ for some $r',c' \geq 0$ also satisfy (1), (2), (3),
  it holds that 
  \[
  \ker \pi_1 \subset \ker \pi_1' \subset \ker (\pi_2' \circ \pi_1' ) \subset \ker (\pi_2 \circ \pi_1 ) \subset \Z^n.
  \]
  In other words, it holds that
  $X_{\pi_2 \circ \pi_1 (A)} \subset X_{\pi_2' \circ \pi_1' (A)} \subset X_{\pi_1'(A)} \subset X_{\pi_1(A)} \subset X_A$
  as subvarieties of $X_A$.
  \end{itemize}
\end{thm}

\begin{rem}\label{rem_uniqueness}
By condition (4), $  \ker \pi_1, \ker (\pi_2 \circ \pi_1 )$ and subvarieties $X_{\pi_2 \circ \pi_1 (A)} \subset  X_{\pi_1(A)} \subset X_A$
 in \autoref{refined_thm_structure} are uniquely determined by $A$.
In other words, $r $ and $c$ are uniquely determined by $A$, and
$\pi_1$ and $\pi_2$ are uniquely determined by $A$ up to $\GL(\Z^{n-c})$ and $\GL(\Z^r)$.

We note that the condition
$\ker (\pi_2' \circ \pi_1' ) \subset \ker (\pi_2 \circ \pi_1 )$ is equivalent to $\pi_2 \circ \pi_1 \preceq \pi'_2 \circ \pi'_1$.
In particular, the Cayley structure $\pi_2 \circ \pi_1 $ is uniquely determined by $A$ up to equivalence.
\end{rem}

We prove \autoref{refined_thm_structure} in the rest of this subsection,
and 
prove \autoref{thm_structure} and \autoref{thm_structure_max} in the next subsection.
In the rest of this subsection,
we fix a finite set $A \subset \Z^n$ with $\langle A-A \rangle =\Z^n$.

If $\delta_{X_A} =0$,
$\pi_1= \id_{\Z^n}$ and the zero homomorphism $\pi_2 : \Z^n \arw \Z^0$ for $r=c=0$ satisfy (1)-(4) in \autoref{refined_thm_structure}
since $\ker \pi_1  = \{0\}$ and $\ker (\pi_2 \circ \pi_1 ) =\Z^n$.
Hence we assume $\delta_{X_A} >0$ in the rest of this subsection.

First,
we construct a Cayley structure $\pi : \Z^n \arw \Z^r$ of $A$ using results in \autoref{sec_refinement}.
Next, we decompose $\pi$ into $\pi_1 : \Z^n \arw \Z^{n-c} $ and ${\pi_2} : \Z^{n-c} \arw \Z^r$ as in the statement of \autoref{refined_thm_structure}
using results in \autoref{sec_contact_locus}.

\begin{lem}\label{lem_find_A_i}
There exists a Cayley structure $\pi : \Z^n \arw \Z^r$ of $A$ for some $r$ 
such that
for a general hyperplane $H \subset \P^N$ which is tangent to $X_A$ at $1_n$,
\begin{itemize}
\item[(i)] the contact locus $Z_H=\overline{\{ x \in (\*c)^n \subset X_A \, | \, H \text{ is tangent to } X_A \text{ at } x \}}$ is contained in the $r$-plane $X_{\pi(A)} \subset X_A$,
\item[(ii)] $\overline{H}_i |_{Z_H} \neq \overline{H}_j |_{Z_H}$ for $i \neq j$, where $\overline{H}_0, \ldots, \overline{H}_r $ are the torus invariant hyperplanes of $X_{\pi(A)} \simeq \P^r$,
\item[(iii)]
 if a Cayley structure $\pi' : \Z^n \arw \Z^{r'}$ of $A$ satisfies (i),  $X_{\pi(A)} \subset X_{\pi'(A)}$ holds.
\end{itemize}
In particular,
such Cayley structure $\pi$ is unique up to equivalence.
\end{lem}

\begin{proof}
Let $\pi : \Z^n \arw \Z^r$ be a Cayley structure of $A$.
Then $\ker \pi $ is generated by $\{ u - u' \, | \, \pi(u) = \pi(u') \}$.
Hence $\ker \pi$ is uniquely determined by the decomposition
\begin{align}\label{eq_decomp}
A = (\pi^{-1}(u'_0) \cap A ) \sqcup (\pi^{-1}(u'_1) \cap A ) \sqcup \cdots \sqcup (\pi^{-1}(u'_r) \cap A ),
\end{align}
where $\pi(A)=\{u'_0,\ldots,u'_r\}$.
Thus as a subvariety,
the $r$-plane $X_{\pi(A)} \subset X_A$ is uniquely determined by the decomposition \ref{eq_decomp}.

Since there are at most finitely many decompositions of $A$ into $r+1$ subsets for $0 \leq r \leq n$,
the following set of subvarieties of $X_A$ is a finite set:
\begin{align}\label{eq_set_of_subvar}
\{ \Lambda \subset X_A \, | \, \Lambda =X_{\pi (A)}  \text{ for a Cayley structure } \pi : \Z^n \arw \Z^r \text{ with } 0 \leq r \leq n  \} .
\end{align}

Let $X_{A, 1_n}^* \subset \Pv$ be the set of hyperplanes which are tangent to $X_A$ at $1_n$.
For general $H \in X_{A, 1_n}^*$,
$Z_H$ is a $\delta_{X_A}$-plane which contains $1_n$.
Applying \autoref{thm_plane_embedding} to $Z_H$,
we have a Cayley structure $\pi_H : \Z^n \arw \Z^{r_H}$ of $A$
as in the statement of \autoref{thm_plane_embedding}.


By \autoref{thm_plane_embedding},
$X_{\pi_H(A)}$ is the minimum subvariety containing $Z_H$ in \ref{eq_set_of_subvar}.
Since \ref{eq_set_of_subvar} is a finite set,
$X_{\pi_H(A)} $ does not depend on the choice of general $H \in X_{A,1_n}^*$,
i.e.,
the equivalence class of the Cayley structure $\pi_H$ does not depend on general $H$.
Let $\pi : \Z^n \arw \Z^r$ be a representative of the equivalence class.
Then (i), (iii) are satisfied since
$X_{\pi (A)}=X_{\pi_H (A)}$ holds for
general $H \in X_{A,1_n}^*$.
By \autoref{rem_minimality}, (ii) follows.
The uniqueness follows from (iii).
\end{proof}


Let $\pi : \Z^n \arw \Z^r$ be the Cayley structure of $A$ obtained by \autoref{lem_find_A_i}.
Set $V:=(\ker \pi)_{\Q} \simeq \Q^{n-r}$.
For $\pi(A)=\{u'_0,\ldots,u'_r\}$, set
\begin{equation}\label{eq_V_i}
  \begin{aligned}
    M_i &:=  \langle \pi^{-1}(u'_i) \cap A  - \pi^{-1}(u'_i) \cap A  \rangle \subset \ker \pi,
    \\
    V_i &:=(M_i)_{\Q} \subset V.
  \end{aligned}
\end{equation}

Assume that $\pi$ is the composite of two surjective group homomorphisms
$\pi_1 : \Z^n \arw \Z^{n-c} $ and ${\pi_2} : \Z^{n-c} \arw \Z^r $ for an integer $c \geq 0$.
Since $\pi_2( \pi_1(A)) =\pi(A) $ is $\Z$-affinely equivalent to $\{0,e_1,\ldots,e_r\}$,
$\pi_2$ is a Cayley structure of $\pi_1(A) \subset \Z^{n-c}$.
As in \autoref{def_join_type_wrt},
the Cayley structure $\pi_2$ is of join type
 if and only if
the subspace $\sum_{i=0}^r M_i^{{\pi_2}} \subset \ker {\pi_2}$
is the inner direct sum of $M_i^{{\pi_2}}$ 's,
where
\[
M_i^{{\pi_2}}:= \langle {\pi_2}^{-1}(u'_i) \cap \pi_1(A)  - {\pi_2}^{-1}(u'_i) \cap \pi_1(A)  \rangle \subset \ker {\pi_2}.
\]
Since $\pi_1 (  \pi^{-1}(u'_i) \cap A  ) = {\pi_2}^{-1}(u'_i) \cap \pi_1(A)$, 
it holds that
\begin{align*}
\pi_1(M_i) &= \langle  \pi_1(\pi^{-1}(u'_i) \cap A ) -  \pi_1(\pi^{-1}(u'_i) \cap A ) \rangle\\
&=\langle {\pi_2}^{-1}(u'_i) \cap \pi_1(A)  - {\pi_2}^{-1}(u'_i) \cap \pi_1(A)  \rangle = M_i^{{\pi_2}}.
\end{align*}
Hence ${\pi_2}$ is of join type if and only if
$\pi_1(M_0) + \dots + \pi_1(M_r) \subset \ker {\pi_2}$ is the inner direct sum of $\pi_1(M_i)$'s.
This is also equivalent to the condition that
${\pi_1}_{\Q}(V_0) + \dots + {\pi_1}_{\Q}(V_r) \subset (\ker {\pi_2})_{\Q}$ is the inner direct sum of ${\pi_1}_{\Q}(V_i)  $'s
for the $\Q$-linear map ${\pi_1}_{\Q} : \Q^n \arw \Q^{n-c}$ obtained from $\pi_1$
since ${\pi_1}_{\Q}(V_i)   = \pi_1(M_i)_{\Q}$.

\vspace{2mm}
Under the above observation,
we construct the homomorphisms $\pi_1$ and ${\pi_2}$ in \autoref{refined_thm_structure} as follows.

\begin{lem}\label{lem_i,j}
Let $\pi : \Z^n \arw \Z^r$ be the Cayley structure obtained by \autoref{lem_find_A_i}.
Let $V_i$ as \ref{eq_V_i} and set $c= \alpha(V_0,\ldots,V_r)$.
Then there exist surjective group homomorphisms $\pi_1 : \Z^n \arw \Z^{n-c}$ and ${\pi_2} : \Z^{n-c} \arw \Z^r$
such that
\begin{itemize}
\item[(i)] $\pi = {\pi_2} \circ \pi_1$,
\item[(ii)] $\pi_2$ is a Cayley structure of $\pi_1(A) \subset \Z^{n-c} $ of join type,
\item[(iii)] if (i) and (ii) are satisfied for $\tilde{\pi}_1 : \Z^n \arw \Z^{n-\tilde{c}}$ and $\tilde{\pi}_2 : \Z^{n-\tilde{c}} \arw \Z^r$,
it holds that $\ker \pi_1 \subset \ker \tilde{\pi}_1$.
\end{itemize}
In addition, we have $\delta_{X_A} = r-c$.
\end{lem}

\begin{proof}

First we assume that $V_0,\dots,V_r$ satisfy the condition $(*)$ in \autoref{key_lem_linear_alg}
and construct $\pi_1$ and ${\pi_2}$ which satisfy (i), (ii), (iii).

Since we assume the condition $(*)$,
we can apply \autoref{key_lem_linear_alg} to $V_i$'s
and obtain a subspace $V' \subset V$. 
Since $V' \subset V=(\ker \pi)_{\Q} \subset (\Z^{n})_{\Q}$ is a $c$-dimensional subspace of $(\Z^{n})_{\Q}$,
$M' := \Z^{n} / (V' \cap \Z^{n})$ is a free abelian group of rank $n-c$.
We fix a group isomorphism $M' \simeq \Z^{n-c}$
and set $\pi_1 : \Z^{n} \arw M' \simeq \Z^{n-r}$ to be the canonical quotient homomorphism.
By construction,
$\ker \pi_1 = V' \cap \Z^{n} \subset \ker \pi$ holds.
Hence $\pi : \Z^n \arw \Z^r$ factors $\pi_1$ as
\[
\pi : \Z^{n} \stackrel{\pi_1}{\longrightarrow} \Z^{n-c} \stackrel{{\pi_2}}{\longrightarrow} \Z^r
\]
with a surjective group homomorphism ${\pi_2} $.
Then (i) holds for these $\pi_1$ and ${\pi_2}$.

Since $V ' $ is obtained by \autoref{key_lem_linear_alg},
\[
{\pi_1}_{\Q}(V_0) + \dots + {\pi_1}_{\Q}(V_r) \subset (\ker {\pi_2})_{\Q} = \pi_1(\ker \pi)_{\Q} \simeq V/V'
\]
is the inner direct sum of ${\pi_1}_{\Q}(V_0) , \dots ,{\pi_1}_{\Q}(V_r) $.
Hence (ii) follows.

If (i) and (ii) are satisfied for $\tilde{\pi}_1 : \Z^n \arw \Z^{n-\tilde{c}}$ and $\tilde{\pi}_2 : \Z^{n-\tilde{c}} \arw \Z^r$,
${\tilde{\pi}{}_1}_{\Q}(V_0) + \dots + {\tilde{\pi}{}_1}_{\Q}(V_r) \subset {(\ker {\tilde{\pi}_2})}_{\Q} $ 
 is the inner direct sum of ${\tilde{\pi}{}_1}_{\Q}(V_i) $'s.
By (2) in \autoref{key_lem_linear_alg},
we have $(\ker \pi_1)_{\Q}=V' \subset (\ker \tilde{\pi}_1)_{\Q} $.
Hence $\ker \pi_1 \subset \ker \tilde{\pi}_1$ holds and (iii) follows.

\vspace{2mm}
To show that $V_0,\dots,V_r$ satisfy the condition $(*)$ in \autoref{key_lem_linear_alg},
let $A_0, \dots, A_r \subset \Z^{n-r}$ be finite subsets obtained by
applying \autoref{lem_projection_cayley} to $\pi$.
Under a suitable group isomorphism $\ker \pi \simeq \Z^{n-r}$,
the subspace $V_i \subset (\ker \pi)_{\Q}$ coincides with $\langle A_i -A_i\rangle_{\Q} \subset (\Z^{n-r})_{\Q}= \Q^{n-r}$ for any $0 \leq i \leq r$.
Hence $V_0,\dots,V_r$ satisfy the assumption in \autoref{key_lem_linear_alg}
if and only if so do $ \langle A_0 -A_0 \rangle_{\Q} , \dots, \langle A_r -A_r \rangle_{\Q}$.
Thus in the rest of the proof of this lemma,
we may assume that $A=A_0 * \cdots * A_r$ for these $A_i$'s
with $V_i = \langle A_i -A_i\rangle_{\Q}$,
and $\pi : \Z^{n-r} \times \Z^r \arw \Z^r$ is the projection to the second factor.

We use the notation of \autoref{lem_tangent_dim} and \autoref{lem_tangent_dim_alpha}.
Let $H \subset \P^N$ be a general hyperplane which is tangent to $X_A$ at $1_n$.
To simplify notation,
we abbreviate $Z_{A,H}^{\circ},Z_{A,H}$ to $Z_H^{\circ}, Z_H$ respectively.
By \autoref{lem_find_A_i},
$Z_H \subset X_{\pi(A)}$ holds.
Since $X_{\pi(A)} = \overline{(\*c)^r} \subset X_A$ under the identification $(\*c)^r \simeq \{1_{n-r}\} \times (\*c)^r \subset  (\*c)^{n-r} \times (\*c)^r $,
we have
\[
Z_{H}^{\circ} = Z_{H}^{\circ} \cap (\C^{\times})^{r} =\left\{ w \in (\C^{\times})^{r} \, \Big| \,  \sum_{0 \leq i \leq r } w_i m_i = 0 \in \C^{n-r} \right\} 
\]
by \autoref{lem_tangent_dim}.
Since $Z_H$ is the closure of $Z_{H}^{\circ} $, it holds that
\begin{align}\label{eq_Z_H=}
Z_{H} =\left\{ [W_0:\cdots:W_r] \in \P^r=X_{\pi(A)} \, \Big| \,  \sum_{0 \leq i \leq r } W_i m_i = 0 \in \C^{n-r} \right\} .
\end{align}
Then we have
\begin{align}\label{delta=r-alpha}
\delta_{X_A} = \dim Z_H =  r -  \dim \langle m_0, \ldots, m_r  \rangle_{\C} =r-c,
\end{align}
where
the first equality follows from the definition of the dual defect,
the second equality follows from \ref{eq_Z_H=},
and the last equality holds since $(m_0,\ldots,m_r) \in K_{\C}$ is general by \autoref{lem_tangent_dim_alpha} and $c=\alpha(V_0,\dots,V_r)$.

Let $\overline{H}_i   \subset X_{\pi(A)}=\P^r$ be the torus invariant hyperplane defined by $W_i=0$.
For any two distinct integers $i, j$ with $0 \leq i, j \leq r$,
it holds that $\overline{H}_i |_{Z_H} \neq \overline{H}_j |_{Z_H}  $ by \autoref{lem_find_A_i} (ii).
Since $ \overline{H}_i |_{Z_H}  , \overline{H}_j |_{Z_H} $ are hyperplanes of $Z_H \simeq \P^{\delta_{X_A}}$, 
\[
\dim \overline{H}_i |_{Z_H}  \cap \overline{H}_j |_{Z_H} = \dim Z_H -2 =  r -  \dim \langle m_0, \ldots, m_r  \rangle_{\C} -2,
\]
where we set $\dim \emptyset = -1$.
On the other hand,
\[
\dim \overline{H}_i |_{Z_H}  \cap \overline{H}_j |_{Z_H} = r -2 - \dim \langle m_l \, | \, 0 \leq l \leq r, l \neq i, j \rangle_{\C}
\] holds
since $\overline{H}_i |_{Z_H}  \cap \overline{H}_j |_{Z_H}  \subset \overline{H}_i \cap \overline{H}_j \simeq \P^{r-2} $ is defined by $ \sum_{l \neq i , j } W_l m_l = 0 \in \C^{n-r}  $.
Hence we have $ \dim \langle m_0, \ldots, m_r  \rangle_{\C} =  \dim \langle m_l \, | \, 0 \leq l \leq r, l \neq i,  j  \rangle_{\C}$.
Thus $(V_0)_{\C},\ldots,(V_r)_{\C}$ satisfy the condition $(*)$ in \autoref{key_lem_linear_alg}.

By a similar argument as \autoref{rem_alpha_R&C},
$V_0,\ldots,V_r$ also satisfy the condition $(*)$ in \autoref{key_lem_linear_alg}.
Hence this lemma is proved.
We note that the last statement of this lemma is already proved in \ref{delta=r-alpha}.
\end{proof}

Now we can prove that the above $\pi_1$ and ${\pi_2}$ satisfy (1)-(4) in \autoref{refined_thm_structure}.

\begin{proof}[Proof of \autoref{refined_thm_structure}]
Let $\pi_1$ and ${\pi_2}$ be the homomorphisms constructed in \autoref{lem_i,j}.
Since ${\pi_2} \circ \pi_1 =\pi$,
the conditions (1) and (2) follows from (i) and (ii) in \autoref{lem_i,j}.
(3) follows from the last statement of \autoref{lem_i,j}.

To check (4), we show the following claim.

\begin{claim}\label{claim_contain_Z}
If surjective group homomorphisms $\pi_1' : \Z^n \arw \Z^{n-c'}$ and $\pi_2' : \Z^{n-c'} \arw \Z^{r'}$ satisfy (1), (2), (3) in \autoref{refined_thm_structure},
the contact locus $Z_{A,H} $ is contained in the $r'$-plane $X_{\pi_2' \circ \pi_1'(A)} \subset   X_A$ for a general hyperplane $H \subset \P^N$ which is tangent to $X_A $ at $1_n$.
\end{claim}

\begin{proof}[Proof of \autoref{claim_contain_Z}]
Let $A'_0,\dots, A'_{r'} \subset \Z^{n-r'}$ be finite subsets obtained by
applying \autoref{lem_projection_cayley} to $\pi_2' \circ \pi_1' : \Z^n \arw \Z^{r'}$.
In particular,
$X_{\pi_2' \circ \pi_1'(A)}  $ coincides with $X_{ \{0,e_1, \dots, e_{r'} \}}$
under the identification $X_A= X_{A'_0 * \dots *A'_{r'}  }$ 

Applying \autoref{lem_tangent_dim_alpha} (iv) to $A'_0 * \dots *A'_{r'}  $,
it holds that
\[
\delta_{X_A } = \dim Z_{A,H} \geq \dim Z^{\circ}_{A,H}  \cap (\*c)^{r'}  \geq r' -c' = \delta_{X_A},
\]
where the last equality follows from condition (3) for $\pi_1',\pi_2'$.
Hence $Z_{A,H} $ coincides with the closure of $Z^{\circ}_{A,H}  \cap (\*c)^{r'} $.
Thus $Z_{A,H}$ is contained in the closure of $ (\*c)^{r'} $, which is nothing but the $r'$-plane $X_{\pi_2' \circ \pi_1'(A)} \subset X_A$.
\end{proof}

Take surjective group homomorphisms $\pi_1' : \Z^n \arw \Z^{n-c'}$ and $\pi_2' : \Z^{n-c'} \arw \Z^{r'}$ for some $r',c' \geq 0$
which satisfy (1), (2), (3) in \autoref{refined_thm_structure}.
By \autoref{claim_contain_Z},
the Cayley structure $\pi':= \pi_2' \circ \pi_1'$ of $A$ satisfies (i) in \autoref{lem_find_A_i}.
Hence $X_{\pi(A)} \subset X_{\pi'(A)}$ holds by \autoref{lem_find_A_i} (iii).
Thus $\ker (\pi_2' \circ \pi_1') =\ker \pi' \subset \ker \pi = \ker (\pi_2 \circ \pi_1)$ holds and 
$\pi :\Z^n \arw \Z^r$ factors $\pi'$ as
\[
\pi : \Z^n \stackrel{\pi'}{\arw } \Z^{r'} \stackrel{\varpi}{\arw} \Z^r
\]
for some $\varpi$.

Since $\ker \pi_1' \subset \ker ( \pi_2' \circ \pi_1') = \ker \pi'$,
the rest is to show $\ker \pi_1 \subset \ker \pi_1'$.
Consider 
\[
\Z^{n-c'} \stackrel{\pi_2'}{\arw} \Z^{r'}  \stackrel{\varpi}{\arw} \Z^r.
\]
Since the Cayley structure $\pi_2'$ of  $\pi_1'(A) \subset \Z^{n-c'} $ is of join type by (2),
so is $\varpi \circ \pi_2'$ by \autoref{lem_pi_pi'}.
By (iii) in \autoref{lem_i,j} for $\tilde{\pi}_1=\pi_1' : \Z^n \arw \Z^{n-c'}$ and $\tilde{\pi}_2 = \varpi \circ \pi_2' : \Z^{n-c'} \arw \Z^r$,
we have $\ker \pi_1 \subset \ker \tilde{\pi}_1 = \ker \pi_1'$.
Thus (4) follows.
\end{proof}

\subsection{Proofs of \autoref{thm_structure} and \autoref{thm_structure_max}}

\autoref{refined_thm_structure} implies \autoref{thm_structure} as follows.

\begin{proof}[Proof of \autoref{thm_structure}]



 

  We take $\pi_1$ and $\pi_2$ as in \autoref{refined_thm_structure}.
  Set $\pi = \pi_2 \circ \pi_1$.
  Applying \autoref{lem_projection_cayley} to $A \subset \Z^n$ and $\pi : \Z^n \arw \Z^r$,
  we have finite subsets $A_0,\ldots,A_r \subset \Z^{n-r}$ such that $A$ is $\Z$-affinely equivalent to $A_0 * \cdots * A_r$.

  Recall the construction of $A_i$.
  We take a section $s : \Z^r \arw \Z^n$ of $0 \arw \ker \pi \arw \Z^n \stackrel{\pi}{\longrightarrow} \Z^r \arw 0$
  and set 
  \[
  A_i = (\pi^{-1}(u'_i) \cap A ) - s(u'_i) \subset \ker \pi \simeq \Z^{n-r},
  \]
  where $\pi(A)=\{u'_0,\ldots,u'_r\}$.

  Since $\pi_1  : \Z^n \arw \Z^{n-c}$ is surjective,
  the restriction of $\pi_1$ on $\ker \pi$ gives a surjective group homomorphism
  $\pi_1 |_{\ker \pi } : \ker \pi  \arw \ker \pi_2 \simeq \Z^{n-c-r}$.
  By the definition of $A_i$,
  \[
  \pi_1(A_i) = ({\pi_2}^{-1} (u'_i) \cap \pi_1(A))  - \pi_1 (s(u'_i)) \subset \ker {\pi_2}
  \]
  holds.
Again, it follows from \autoref{lem_projection_cayley} that
$\pi_1(A)$ is $\Z$-affinely equivalent to $\pi_1(A_0) * \cdots * \pi_1(A_r) $ 
  since $\pi_1 \circ s $ is a section of ${\pi_2} : \Z^{n-c} \arw \Z^r$.
  By the condition (2),
  $\pi_1(A_0) * \cdots * \pi_1(A_r) $ is of join type.
  Hence $A_0,\ldots,A_r \subset \Z^{n-r}$ and $p   : \Z^{n-r} \simeq \ker \pi  \stackrel{\pi_1 |_{\ker \pi }}{\longrightarrow} \ker {\pi_2} \simeq \Z^{n-c-r} $
  satisfy (a) and (b).
Conditions (c), (d) follow from (3), (4) respectively.
\end{proof}

\autoref{thm_structure_max} follows from \autoref{lem_tangent_dim_alpha} (iv) and \autoref{thm_structure}.

\begin{proof}[Proof of \autoref{thm_structure_max}]
By \autoref{lem_tangent_dim_alpha} (iv),
$\delta_{X_A}$ is at least the maximum of the values $r-c$ in the statement of \autoref{thm_structure_max}.
The opposite inequality follows from \autoref{thm_structure}.
\end{proof}

\subsection{Examples}

In the following examples, 
we assume that $A \subset \Z^n$ spans $\Z^n$ as an affine lattice.

\begin{ex}\label{ex_delta=0}
If $\delta_{X_A} =0$,
we can take $\pi_1 =\id_{\Z^n} : \Z^n \arw \Z^n$ and ${\pi_2} : \Z^n \arw \Z^{0} $ with $r=c=0$.

As $A_i$ and $p$ in \autoref{thm_structure},
we can take $r=c=0$, $A_0=A$, and $p=\id_{\Z^n}$ 
since $A$ is $\Z$-affinely equivalent to the Cayley sum $A \times \{0\} \subset \Z^n \times \Z^{0}$ of join type.
\end{ex}

\begin{ex}\label{ex_in_intro_again}
Let $A=\{(0,0,0), (1,0,0), (0,1,0), (1,1,0), (0,0,1)\} \subset \Z^3$ as in \autoref{ex_not_unique}.
In this case,
$\pi_1$ and $\pi_2$ are the identity $\id_{\Z^3} : \Z^3 \arw \Z^3$ and the projection $\pr_3 : \Z^3 \arw \Z$ to the third factor
respectively
with $r=1,c=0$.
\end{ex}

\begin{ex}\label{subsec_ex}
Let $A \subset \Z^2 \times \Z^3$ be the Cayley sum of $A_0,\ldots,A_3 \subset \Z^2$ for
\begin{align*}
&A_0=\{(0,0),(1,0), (2,0)\}, \quad A_1 =\{ (0,0), (0,1), (0,2)\},\\
&\hspace{15mm} A_2=A_3 =\{(0,0),(1,0),(0,1),(1,1)\} .
\end{align*}

If there exists a surjective group homomorphism $\pi' :  \Z^2 \times \Z^3 \arw \Z^{r'}$
such that $\pi'(A)$ is $\Z$-affinely equivalent to $\{0,e_1,\ldots,e_{r'}\}$,
then $\pi'(A_0 \times \{0\} )$ consists of one point.
In fact, for $ l=0,1,2$ and $\bar{u}_l = \pi' ((l,0),0) \in \pi'(A_0 \times \{0\}) $,
it holds that $\bar{u}_0 - 2 \bar{u}_1 + \bar{u}_2=0$.
Since $\pi'(A)$ consists of $\Z$-affinely independent $r'+1$ points,
$\bar{u}_0=\bar{u}_1=\bar{u}_2$ must hold.
Similarly,
$\pi'(A_1 \times \{e_1\})$ consists of one point as well. 
Hence $\pi'$ factors through the projection $\pr : \Z^2 \times \Z^3 \arw \Z^3$ to the second factor
(we use the notation $\pr$ to distinguish $\pi ={\pi_2} \circ \pi_1$ in the proof of \autoref{refined_thm_structure},
although we will see $\pr={\pi_2} \circ \pi_1$ in the following paragraphs).

Take $\pi_1, {\pi_2} $ in \autoref{refined_thm_structure} for this $A$.
By the previous paragraph,
${\pi_2} \circ \pi_1$ factors through the projection $\pr : \Z^2 \times \Z^3 \arw \Z^3$.
In fact,
we can take $\pi_1=\pr$ and ${\pi_2} = \id_{\Z^3}$ with$(r,c) = (3,2)$, and hence $\delta_{X_A}=3-2=1$.
The reason is as follows.


By the construction of $\pi_1,{\pi_2}$,
$Z_H$ is contained in $X_{{\pi_2} \circ \pi_1 (A)}$
for general $H$ which is tangent to $X_A$ at $1_5=(1_2,1_3)$.
Since $\pi_2 \circ \pi_1$ factors through $\pr$,
we have $X_{{\pi_2} \circ \pi_1 (A)} \subset X_{\pr(A)}$.
Hence $Z_H$ is contained in $X_{\pr(A)} \simeq \P^3$.
Applying \autoref{lem_tangent_dim} and \autoref{lem_tangent_dim_alpha} to $A_0 * A_1 * A_2 *A_3$,
we have the following description of $Z_H = \overline{Z_H^{\circ} \cap (\*c)^3}$ :
Since $V_0 =\Q (1,0),V_1=\Q (0,1), V_2=V_3= \Q^2$,
 $(m_0,m_1,m_2,m_3) \in K_{\C}$ induced by $H$ is written as $((-a-a',0), (0,-b-b') ,(a,b),(a',b'))$ for general $a,b,a',b' \in \C$.
Hence
\[
Z_H =\left\{ [W_0:W_1:W_2:W_3] \in \P^3 \, | \, W_0 m_0 + W_1 m_1 + W_2 m_2 +W_3 m_3=0 \right\} \subset \P^3
\]
is the line defined by $(-a-a')W_0 +a W_2 + a'W_3 = (-b-b')W_1 + b W_2 +b' W_3 =0 $.
Since $Z_H$ is contained in $X_{{\pi_2} \circ \pi_1 (A)} \subset X_{\pr(A)} =\P^3$ for any general $a,b,a',b'$,
it must hold that $X_{{\pi_2} \circ \pi_1 (A)} =X_{\pr(A)}$,
i.e.,
$\ker {\pi_2} \circ \pi_1 = \ker \pr$.
Hence we may assume ${\pi_2} \circ \pi_1 =  \pr$.
For these $V_i$'s,
we have $\alpha=\dim_{\C} \langle (-a-a',0), (0,-b-b') ,(a,b),(a',b') \rangle =2$.
This means that $c=\rank (\ker \pi_1)=2$.
Since $\ker \pi_1 \subset \ker {\pi_2} \circ \pi_1 =\ker \pr = \Z^2 \times \{0\}$,
we have $\ker \pi_1 = \Z^2 \times \{0\}$. 
Thus we can take $\pr : \Z^2 \times \Z^3 \arw \Z^3$ as $\pi_1$ and
hence $\pi_2 =\id_{\Z^3}$.
We will see this example again in the next section.
\end{ex}

\begin{ex}

In the above examples,
$\pi_1$ or $\pi_2$ is the identity homomorphism.
As an example with $\pi_1 \neq \id_{\Z^n},\pi_2 \neq \id_{\Z^r}$,
we consider 
\[
A =\{0,e_1,e_2,e_3,e_4,e_4,e_5,e_6, u,v\} \subset \Z^6,
\]
where $u=(-1,2,0,0,-2,1), v = (0,0, -1,2,-2,1)$.
We note that this example is nothing but Example 6.4 in \cite{BDR} up to a change of indices of $e_1,\dots,e_6$ (see also \cite[Section 4]{DN}).
By \cite[Example 6.4]{BDR}, $X_A \subset \P^8$ is strongly self dual,
i.e., $X_A = X^*_A$ holds in a suitable sense (see \cite{BDR} for self duality). 
In particular, we have $\delta_{X_A} =1$.

We define $\Z^6 \stackrel{\pi_1}{\longrightarrow} \Z^5 \stackrel{\pi_2}{\longrightarrow} \Z^2$ by
\[
(x_1,\dots, x_6 ) \mapsto (x_1,x_2,x_3,x_4,x_5 +2 x_6) \mapsto (x_1 + x_2, x_3+x_4).
\] 
For $\pi := \pi_2 \circ \pi_1$,
we have $\pi(A) = \{(0,0),(1,0),(0,1)\}$.
Let $A^0,A^1,A^2 \subset A$ be the fibers of $\pi |_{A}  $ over $(0,0),(1,0),(0,1)$ respectively.
Then 
\[
A^0 =\{0,e_5,e_6\}, \quad  A^1 =\{e_1,e_2,u\}, \quad  A^2=\{e_3,e_4,v\}.
\]
For the standard basis $f_1,\dots,f_5$ of $\Z^5$,
it holds that
\[
\pi_1(A^0) =\{0,f_5,2 f_5\}, \ \ \pi_1(A^1) = \{f_1,f_2, - f_1 +2f_2\}, \ \ \pi_1(A^2) = \{f_3,f_4 , - f_3 +2f_4\}
\]
and hence $\langle \pi_1(A^i) -\pi_1(A^i) \rangle \subset \ker \pi_2$ for $i=0,1,2$ are
\[
 \langle f_5 \rangle,   \quad  \langle  f_2-f_1 \rangle ,
 \quad  \langle f_4 -f_3\rangle
\]
respectively.
Since $ \langle f_5 \rangle +  \langle  f_2-f_1 \rangle +\langle f_4 -f_3\rangle \subset \ker \pi_2 \subset \Z^5$ is the inner direct sum,
$\pi_2$ is a Cayley structure of $\pi_1(A)$ of join type.
Hence $\pi_1,\pi_2$ satisfy conditions (1), (2) in \autoref{refined_thm_structure}.
Condition (3) also holds since $r=2, c=1$ and $\delta_{X_A} = 1$.
If condition (4) is not satisfied for these $\pi_1,\pi_2$,
there exists surjective group homomorphisms $\Z^6 \arw \Z^{6-c} $ and $\Z^{6-c} \arw \Z^{r}$ for $r=1, c=0$
which satisfy (1), (2).
This means that $A $ is $\Z$-affinely equivalent to a Cayley sum $B_0 * B_1 $ of join type.
We can check that there are no such $B_i$'s, and hence condition (4) is satisfied for these $\pi_1,\pi_2$.
%
\end{ex}

\section{Geometric description}\label{sec_geometry}

In this section,
we interpret \autoref{thm_structure} geometrically.

\subsection{Joins and toric varieties}\label{subsec_join}

Let $X_0,\ldots,X_r \subset \P^N$ be projective varieties
such that $ \overline{x_0  \dots x_r} $ is $r$-dimensional for general $(x_0, \dots, x_r) \in X_0 \times \dots \times X_r$,
where $\overline{x_0 \dots x_r} \subset \P^N$ is the linear subvariety
spanned by $r+1$ points $x_0,  \dots, x_r$
(in this paper, we always assume this condition when we consider joins of varieties).

The \emph{join} $J(X_0, \dots, X_r) \subset \PN$ of $X_0, \dots, X_r \subset \PN$ is defined to be
\[
J(X_0, \dots, X_r) = \bigcup_{x_0, \in X_0, \dots, x_r \in X_r} \overline{x_0  \dots x_r}  \subset \PN.
\]
Then we have $\dim J(X_0, \dots, X_r) \leq r + \sum_{i=0}^r \dim X_i$ and $\delta_{ J(X_0, \dots, X_r)} \geq r$.
\\

Now we consider a Cayley sum of join type defined in \autoref{def_join_type}.
It indeed gives a join of toric varieties.
Moreover, it has the following simple structure.

Let $A_0,\ldots,A_r \subset \Z^{n-r}$ be finite sets and let $A=A_0 * \cdots * A_r  $ be the Cayley sum.
We set $\Lambda_i = \CC^{\# A_i}$ for $0 \leq i \leq r$
and
$\Lambda = \CC^{\# A} = \Lambda_0 \oplus \dots \oplus \Lambda_r$.
Then we have toric varieties $X_{A_i} \subset \cP(\Lambda_i) \subset \cP(\Lambda)$.
We note that $X_{A_i} \subset \cP(\Lambda)$ coincides with the torus invariant subvariety of $X_A \subset \cP(\Lambda)$
corresponding to the face $\Conv(A_i \times \{e_i\}) $ of $\Conv(A)$.

Consider the join
\[
J(X_{A_0},\ldots,X_{A_r}) \subset \cP(\Lambda) = \cP(\Lambda_0\oplus \dots \oplus \Lambda_r).
\]
The structure of this join is simple; for example,
$\dim J(X_{A_0},\ldots,X_{A_r}) = r+\sum_{i=0}^r \dim X_{A_i}$
and $X_{A_i} \cap X_{A_j} = \emptyset$ if $i \neq j$.
The following equality also holds.

\begin{lem}\label{del_J=r+del_i}
In the above setting,
$\delta_{J(X_{A_0},\ldots,X_{A_r})} = r + \sum_{i=0}^r \delta_{X_{A_i}}$.
\end{lem}
\begin{proof}
  Consider the case of $r = 1$. Write $X := X_A \subset \PN$ and $X_{i} := X_{A_i} \subset \PN$.
  For a general hyperplane $H$ which is tangent to $X$,
  the contact locus $Z_{X, H} $ coincides with the join $ J(Z_{X_0, H}, Z_{X_1, H})$ due to Terracini's lemma.
  Hence $\delta_X = \delta_{X_0}+\delta_{X_1}+1$ holds.
  Inductively, we have the assertion.
\end{proof}



The following lemma gives a geometric meaning of Cayley sums of join type.

\begin{lem}[{\cite[Lemma 4.10]{FI}}]\label{em_join _type=join}
  Let $A_0,\ldots,A_r \subset \Z^{n-r}$ be finite sets and let $A=A_0 * \cdots * A_r  $ be the Cayley sum.
  Then $X_A \subset \cP(\Lambda) $ is contained in $ J(X_{A_0},\ldots,X_{A_r})$,
and  $X_A = J(X_{A_0},\ldots,X_{A_r})$ holds if and only if the Cayley sum $A = A_0 * \cdots * A_r$ is of join type.
\end{lem}

If a toric variety $X_A$ is a join of some (not necessarily toric) varieties, then
$A$ is in fact a Cayley sum of join type,
and hence $X_A$ has the above structure.
This follows from the following result.
(Since the statement of {\cite[Corollary 4.14]{FI}} is not so clear, we restate it.)

\begin{lem}
  Let $A \subset \ZZ^n$ be a finite set with $\lin{A-A} = \ZZ^n$.
  Assume that the toric variety $X_A \subset \PN$ is a join of $l+1$ projective varieties
  $X_0, \dots, X_l$
  (as we noted in the first paragraph of this subsection,
  we assume that $ \dim \overline{x_0 \cdots x_l} =l$ holds 
  for general $(x_0, \dots, x_l) \in X_0 \times \dots \times X_l$).
  Then there exist
  finite subsets $A_0,\dots,A_r \subset \ZZ^{n-r}$
  with $r \geq l$ such that
  $A$ is $\Z$-affinely equivalent to the Cayley sum $A_0 * \dots * A_r$, which is of join type.
\end{lem}
\begin{proof}
  Let $(x_0, \dots, x_l) \in X_0 \times \dots \times X_l$ be general.
  By Terracini's lemma, the embedded tangent space $\TT_xX \subset \P^N$ is constant on
  $x \in \overline{x_0 \cdots x_l} \cap X_{sm}$.
  Hence the Gauss defect $r := \delta_{\gamma}(X_A) \geq \l$ (see \cite{FI} for the definition of Gauss defects).
  Then
  $A$ is $\Z$-affinely equivalent to a Cayley sum $A_0* \cdots * A_r$ of join type for some $A_0, \dots, A_r$ by {\cite[Corollary 4.13]{FI}}.
\end{proof}

\subsection{Geometric description of dual defects of toric varieties}\label{subsec_geometric_description}




We denote by $C(X) \subset X \times \Pv$
the closure of $\set*{(x, H) \in X_{sm} \times \Pv}{\TT_xX \subset H}$,
which is called the \emph{conormal variety} of $X$.
Set $\pr_i$ to be the projection from $C(X)$ to the $i$-th factor for $i = 1,2$.
By definition, $ \pr_2(C(X)) \subset \Pv$ is the dual variety $X^*$ of $X$.
We note that $\pr_1(\pr_2^{-1}(H)) \subset X$
is nothing but the contact locus $Z_H$ of $H$ on $X$ for general $H \in X^*$.
For a subvariety $Y \subset X$, we denote by $C(X)|_Y \subset Y \times \Pv$
the closure of $C(X_{sm}) \cap (Y \times \Pv)$.

The inequality in the following lemma was already shown in \cite[Proposition 3.5]{LS}.
Later we need the statement in the equality case.

\begin{lem} \label{lem_covering_family-plus} \label{lem_covering_family}
  Let $X \subset \PN$ be a projective variety.
  Take a general $H \in X^*$,
  and let $Y \subset X$ be a closed subvariety
  of codimension $c$
  such that $Y_{sm} \cap Z_H \cap X_{sm} \neq \emptyset$.
  Then we have $\delta_X \geq \delta_Y-c$.
  Moreover, if $\delta_X = \delta_Y-c$ holds, then we have $Y^* = \pr_2(C(X)|_Y) \subset X^*$ and
  $\pr_2^{-1}(H) = \pr_2^{-1}(H) \cap C(X)|_Y$; in particular,
  the contact locus on $X$ of $H$ is contained in $Y$.
\end{lem}
\begin{proof}
  Let us consider
  \[
  \xymatrix{%
    C(X)|_Y \ar@{}[r]|{\mbox{$\subset$}} \ar[d] & C(X) \ar[r]^{\kern-1.5em\pr_2} \ar[d]^{\pr_1} & X^* \subset \Pv
    \\
    Y \ar@{}[r]|{\mbox{$\subset$}} & X \makebox[0pt]{\,.}
  }%
  \]
  Then $\pr_2(C(X)|_Y) \subset Y^*$ holds since
  $\TT_{x'}Y \subset \TT_{x'}X \subset H$ for general $(x', H') \in C(X)|_Y$.
  Since $(x,H) \in C(X)|_Y $ for $x \in Y_{sm} \cap Z_H \cap X_{sm}$,
  we have $H \in Y^*$.
  Since
  $Z_H =\pr_1(\pr_2^{-1}(H)) \simeq \pr_2^{-1}(H)$,
  we have
    \begin{equation}\label{eq:deltaX-geq-deltaY-c}
    \begin{aligned}
    \delta_X &= \dim (\pr_2^{-1}(H)) \\ &\geq \dim (\pr_2^{-1}(H) \cap C(X)|_Y)
    \\
    &\geq \dim C(X)|_Y  - \dim \pr_2(C(X)|_Y) \\
    &\geq \dim C(X)|_Y  - \dim Y^*
 \\
    &= (N-1-\dim X+\dim Y ) - (N-1-\delta_Y)
    = \delta_Y - c.
    \end{aligned}
  \end{equation}


  Next assume that $\delta_X = \delta_Y - c$.
  Then the equality holds in the first inequality of \ref{eq:deltaX-geq-deltaY-c},
  hence we have $\pr_2^{-1}(H) = \pr_2^{-1}(H) \cap C(X)|_Y$.
  By the third inequality, we also have $\dim \pr_2(C(X)|_Y) = \dim Y^* $;
  hence $Y^* = \pr_2(C(X)|_Y) \subset X^*$.
\end{proof}

If $\{Y_{s }\}_{s}$ is a covering family of $X$,
we can take some $Y_s$ as $Y$ in \autoref{lem_covering_family} for general $H$.
Hence we have $\delta_X \geq \delta -c$ as stated in Introduction,
where $c$ is the codimension of $Y_s$ and $\delta = \delta_{Y_{s}} $ for general $s$.

Geometrically,
\autoref{lem_tangent_dim_alpha} (iv) is a special case of the following lemma.

\begin{lem}\label{lem_sub_join}
Let $A \subset \Z^n$ be a finite subset with $\langle  A - A  \rangle = \Z^n  $.
Let $Y \subset X_A$ be a subvariety of codimension $c$ with $ Y \cap (\*c)^n \neq \emptyset$.
Then $\delta_{X_A} \geq \delta_{Y} -c$ holds.

In particular,
if $Y$ is the join of $r+1$ projective varieties, 
then 
$\delta_{X_A} \geq r-c$ holds.
\end{lem}

\begin{proof}
By the torus action,
we have an isomorphism $t \, \cdot : X_A \arw X_A$ for each $t \in (\*c)^n $.
Let $t \cdot Y \subset X_A$ be the image of $Y$ by this isomorphism.
Then $\delta_{t \cdot Y} =\delta_Y$ holds for any $t \in (\*c)^n$.
Since $ Y \cap (\*c)^n \neq \emptyset$,
$\{ t \cdot Y\}_{t \in (\*c)^n }$ covers $X_A$.
Hence we have $\delta_{X_A} \geq \delta_{Y} -c$. 
Since the dual defect of the join of $r+1$ varieties is at least $r$, we have the last statement.
\end{proof}

We see one more lemma.

\begin{lem}\label{lem_A_i:non-defective}
Let $A \subset \Z^n$ be a finite subset with $\langle  A - A  \rangle = \Z^n  $.
Take $A_0,\dots,A_r \subset \Z^{n-r}$ and $p : \Z^{n-r} \arw \Z^{n-r-c}$ which satisfy (a),(b),(c) in \autoref{thm_structure}.
Then $\delta_{X_{p(A_i)}}=0$ holds for each $0 \leq i \leq r$.
In particular,
$\delta_{J(X_{p(A_0)},\ldots,X_{p(A_r)})} = r $ holds.
\end{lem}

\begin{proof}
Although we can show this lemma by applying \autoref{thm_structure} to each $A_i$,
we give a more geometric proof using the previous lemmas in this section.

As $Y$ in \autoref{lem_sub_join},
we take $X_{ p \times \id_{\Z^r} (A_0 * \cdots * A_r) } = X_{ p(A_0) * \cdots * p(A_r)}$ in $X_A$.
Since $p(A_0) * \cdots * p(A_r) $ is of join type,
$X_{ p(A_0) * \cdots * p(A_r)}$ is the join $J(X_{p(A_0)},\ldots,X_{p(A_r)})$ by \autoref{em_join _type=join}.
Then we have
\[
r-c =\delta_{X_A} \geq \delta_{J(X_{p(A_0)},\ldots,X_{p(A_r)})} -c = r + \sum_{i=0}^r \delta_{X_{p(A_i)}} -c,
\]
where the middle inequality follows from \autoref{lem_sub_join}
and the last equality follows from \autoref{del_J=r+del_i}.
Hence each $ \delta_{X_{p(A_i)}}$ must be $0$.
\end{proof}

Casagrande and Di Rocco \cite[Corollary 5.5, Remark 5.6]{CD} proved that 
if $X_A$ is normal and $\Q$-factorial,
there exists an elementary extremal contraction of fiber type 
$\psi : X_A \arw X'$ whose fiber is the join of $r+1$ toric varieties and $\delta_{X_A} = r - \dim X'$.
We note that $\dim X'$ is nothing but the codimension of the fiber in $X_A$.  
As an analog of their result for any $X_A$,
we can show \autoref{thm_structure_geometry},
which states that there exists a dominant rational map $\psi : X_A \dashrightarrow (\*c)^c$ for some $c \geq 0$
such that the closure of each fiber is 
the join of $r+1$ non-defective toric varieties and $\delta_{X_A} = r-c$:


\begin{proof}[Proof of \autoref{thm_structure_geometry}]
We take $A_0,\ldots,A_r \subset \Z^{n-r}$ and $ p : \Z^{n-r} \arw \Z^{n-r-c}$ in the statement of \autoref{thm_structure}.
We may assume that $A=A_0 * \cdots * A_r$.

Since $\ker (p \times \id_{\Z^r}) \simeq  \Z^{c}$,
the inclusion $\ker (p \times \id_{\Z^r})  \subset \Z^{n-r} \times \Z^r $ induces a surjective morphism $\psi : (\*c)^{n-r} \times (\*c)^r \arw (\*c)^c$ between algebraic tori.
We regard $\psi$ as a rational map from $X_A $ to $ (\*c)^c$.
Then the closure of the fiber $\psi^{-1}(1_c)$ in $X_A$ is nothing but $X_{ p \times \id_{\Z^r} (A) } =J(X_{p(A_0)},\ldots,X_{p(A_r)})  \subset X_A$.
By \autoref{lem_A_i:non-defective},
we have $ \delta_{X_{p(A_i)}}=0$, i.e.,
$X_{p(A_i)}$ is non-defective.
Hence this $\psi$ satisfies the condition of this theorem.
\end{proof}

Let us consider 
the dual variety $X_A^* \subset \Pv$ of the toric variety $X_A$.
Since the projective duality $(X_A^*)^* = X_A$ holds in characteristic zero,
we have $\delta_{X_A^*} = N-1-n$.
In particular,
$X_A^*$ is always defective if $X_A$ is not a hypersurface.


As the dual of $\psi: X_A \dashrightarrow \Cc$,
a fibration structure of $X_A^* \subset \Pv$
is given 
in the following sense.
When $X_A$ is normal and $\Q$-factorial,
this was proved in \cite[Theorem 5.2]{CD}.

\begin{prop}
  Let $X_A \subset \PN$ and $\psi$ be as in \autoref{thm_structure_geometry}.
  Then there exists a dominant rational map 
  $\psi^*: X_A^* \dashrightarrow \Cc$
  such that
  $\overline{(\psi^*)^{-1}(s)} = (\overline{\psi^{-1}(s)})^*$ in $\Pv$
  and $\delta_{X_A^*} = \delta_{\overline{(\psi^*)^{-1}(s)}} - c$
  for general $s \in \Cc$.
\end{prop}

\begin{proof}
  We write $X = X_A$.
  For $s\in \Cc$, set $Y_s = \overline{\psi^{-1}(s)} \subset X$.
  From \autoref{thm_structure_geometry} and \autoref{lem_A_i:non-defective}, we have $\delta_X = r-c = \delta_{Y_s} - c$.
  From \autoref{lem_covering_family-plus}, the contact locus on $X$ of a general $H \in X^*$
  is contained in some $Y_s$.
  Thus this proposition can be shown in a similar way to 
  \cite[Proposition 5.1]{CD}; however our $\psi$ is not a morphism and we need to take care of
  this point.

  We take $B$ to be a projective variety containing $\Cc$ as an open subset.
  By resolving the indeterminacy of the rational map 
  $\psi \circ \pr_1 : C(X) \rightarrow X \dashrightarrow \Cc \subset B$,
  we take a birational map $\mu : \widetilde {C}(X)  \rightarrow C(X)$ for some normal projective variety $ \widetilde {C}(X) $
  such that $\Psi :=\psi \circ \pr_1 \circ \mu $ is a morphism from $\widetilde{C}(X) $ to $B$.
  Set $\widetilde{\pr}_2 : = \pr_2 \circ \mu : \widetilde {C}(X)  \arw C(X) \arw X^* \subset \Pv$.
\[
\xymatrix{
\widetilde{C}(X) \ar[r]^{\mu} \ar@/_6mm/[rdd]_(.4){\Psi} \ar@/^7mm/[rr]^{\widetilde{\pr}_2}  & C(X) \ar[d]^{\pr_1} \ar[r]^{\pr_2} & X^*  \\ 
& X \ar@{-->}[d]^{\psi}  &   \\
&  B \supset (\*c)^{c} &   \\
}\]
  Since $\widetilde{C}(X) $ is normal,
  $\widetilde{\pr}_2$ factors through the normalization $\nu : (X^*)^{nor} \arw X^*$,
  and we have a morphism $\widetilde{\pr}'_2 : \widetilde{C}(X) \arw (X^*)^{nor} $.
  Since $\pr_2$ is a projective bundle over $(X^*)_{sm}$ and $ \mu$ is birational,
  $\widetilde{\pr}_2 $ has connected fibers over $(X^*)_{sm}$.
  Hence $\widetilde{\pr}'_2 $ has connected fibers.

  Since $Y_s$'s cover $X$, $C(X)|_{Y_s}$'s cover $C(X)$.
  From \autoref{lem_covering_family-plus},
  we have $\pr_2(C(X)|_{Y_s}) = Y_s^*$ and hence $Y_s^*$'s cover $X^*$.
  Thus, 
  for general $H \in X^*$, we can take some $s \in \Cc \subset B$ 
  such that $H \in Y_s^*$. 
  Then $\pr_2^{-1}(H) \subset C(X)|_{Y_s}$ holds by \autoref{lem_covering_family},
  which implies $\widetilde{\pr}_2^{-1}(H) \subset \Psi^{-1}(s)$.
  In other words, a general fiber of $\widetilde{\pr}_2$, which is also a general fiber of $\widetilde{\pr}'_2 $, is contracted to a point by $\Psi$.

Applying \cite[Lemma 1.15 (a)]{De} to $\widetilde{\pr}'_2 : \widetilde{C}(X) \arw (X^*)^{nor} $ and $\Psi$,
  we have a rational map $\psi^*: X^* \dashrightarrow B$
  such that $\Psi = \psi^* \circ \widetilde{\pr}_2$ holds as rational maps.
  Then $\overline{(\psi^*)^{-1}(s) }= \widetilde{\pr}_2(\Psi^{-1}(s)) = \pr_2(C(X)|_{Y_s}) = Y_s^*$ holds for general $s $.
  From the projective duality $X^{**} = X$ and $Y_s^{**} = Y_s$,
  we have $\delta_{X^*} = N-1 - \dim X = N-1 - (\dim Y_s + c) = \delta_{Y_s^*} -c$.
\end{proof}

\begin{ex}\label{ex2}
Contrary to the $\Q$-factorial case,
we cannot take a morphism as $\psi$ in \autoref{thm_structure_geometry} in general,
even if $X_A $ is normal.

To see this,
consider $A= A_0 * A_1 * A_2*A_3 \subset \Z^2 \times \Z^3$ in \autoref{subsec_ex}.
By  \autoref{subsec_ex}, we have $r=3, c=2$, and hence $\delta_{X_A} =1$.

By the definition of $A_i$,
there exists a birational morphism
$\mu : \P_{\P^1 \times \P^1} (\cale )  \arw X_A \subset \P^{13}$,
where $\cale:=\calo(2,0) \oplus \calo(0,2) \oplus \calo(1,1) \oplus \calo(1,1)$.
The contracted locus of $\mu$ is the disjoint union of two sections of $\P_{\P^1 \times \P^1}(\cale) \arw \P^1 \times \P^1$ corresponding to $\calo(2,0)$ and $ \calo(0,2) $.
Thus the Picard number of $X_A$ is one.
Hence there is no surjective morphism from $X_A$ to a two dimensional variety. 
In this case,
$\psi$ is the rational map
\[
X_A \stackrel{\mu^{-1}}{\dashrightarrow} \P_{\P^1 \times \P^1} (\cale )  \arw \P^1 \times \P^1 \supset (\*c)^2.
\]

Since the natural map $\Sym^d H^0(\cale) \arw H^0(\Sym^d \cale)$ is surjective for any $d \in \N$,
the embedding $ X_A \subset \P^{13}$ is projectively normal.
In particular,
$X_A$ is normal.
\end{ex}

\end{document}